\documentclass[11pt,a4paper,twoside]{article}
\usepackage{amsfonts} 
\usepackage{amssymb}
\usepackage{amsmath, amsthm, array, ifthen}
\usepackage{mathrsfs} 
\usepackage{caption} 
\usepackage{latexsym}
\usepackage{enumerate}

\usepackage{enumitem}

\usepackage{xfrac}
\usepackage{graphicx}

\usepackage{lmodern}

\usepackage[french]{babel} 
\usepackage[utf8]{inputenc} 

\usepackage{color}

\usepackage[colorlinks]{hyperref}

\newcommand{\RR}{\mathbb{R}}
\newcommand{\CC}{\mathbb{C}}

\newcommand{\epsi}{\varepsilon}

\newcommand{\mell}[2][f]{\mathcal{M}(f)} 

\newcommand{\diff}{\mathop{}\mathopen{}\mathrm{d}}

\newtheorem{prop}{Proposition}
\newtheorem{prop/not}{Proposition/notation}
\newtheorem{thm}{Théorème}

\newtheorem{cor}{Corollaire}
\newtheorem{lem}{Lemme}

\theoremstyle{definition}
\newtheorem{rem}{Remarque}

\newtheorem{thmintro}{Théorème}

\newcommand{\ioe}{\leqslant}
\newcommand{\soe}{\geqslant}

\usepackage[usenames,dvipsnames,svgnames,table]{xcolor}

\title{ Conversions explicites entre des fonctions sommatoires de la fonction de M\"obius}

\usepackage[left=3cm,right=3cm,top=1.5cm,bottom=1.5cm]{geometry}
\author{Daval Florian}
\date{}
\begin{document}
\maketitle
\begin{abstract} En utilisant exclusivement les outils de l'analyse réelle, nous fournissons des estimations explicites de plusieurs fonctions sommatoires classiques faisant intervenir la fonction de M\"obius, 
et en donnons quelques applications. 

Nous partons des meilleures estimations explicites disponibles de la fonction sommatoire $M(x)=\sum_{n \ioe x} \mu(n)$ du type $|M(x) / x| \ioe A_j /(\log x)^j$ pour tout $x \soe T_j$ avec $j=0, 1$ et $2$ pour obtenir des encadrements de $m(x)=\sum_{n \ioe x} \mu(n)/n$ du même type, soit  $|m(x)|  \ioe A'_j /(\log x)^j$ pour tout $x \soe T'_j$ avec $A'_j=A_j(1+\epsi_j)$ et $0<\epsi_j<6 \times 10^{-4}$ qui améliorent les résultats antérieurs.

Ensuite, nous étudions la fonction $\check{m}(x)=\sum_{n \ioe x} (\mu(n)/n) \log(x/n)$ et prouvons notamment  que  $|\check{m}(x)-1| \ioe A''_j/(\log x)^j $ avec $A''_j< 1.2 \epsi_j$. En particulier on obtient la valeur exacte de $\sup_{x \soe 1} (\log^2 x) |\check{m}(x)-1|$.

Enfin, à partir de l'estimation $\varlimsup  |M(x)/\sqrt{x}|>1.837625 $ obtenue par  Hurst, nous démontrons que 
$\varlimsup  |m(x) \sqrt{x}| > \sqrt{2}$, ce qui montre en particulier  que le supremum de  $|m(x) \sqrt{x}|$ n'est pas atteint pour $x \to 2^{-}$ comme pourraient le laisser croire les premières observations numériques.
\end{abstract}

\section{Introduction}
Par divers aspects, la fonction de  M\"obius $\mu(n)$, définie formellement par $\sum \mu(n) n^{-z}=(\sum  n^{-z})^{-1}$, est liée aux nombres premiers. Nous allons nous appuyer sur des estimations déjà obtenues par d'autres auteurs  pour la fonction sommatoire $\sum_{n \ioe x} \mu(n)$ et  voici quatre conséquences de l'approche que nous développons dans cet article.
\begin{thmintro} \label{thm A}
\begin{equation*}
|\sum_{n \ioe x} \mu(n)/n | \ioe
\begin{cases}
1/4343  \quad &(x \soe 2\,160\,605)\,,\\
0.0130073/ \log x \quad &(x\soe 97\,063)\,,\\
362.84/(\log x)^2 \quad &(x\soe 2 )\,. 
\end{cases} 
\end{equation*}
\end{thmintro}
\begin{thmintro} \label{thm B}
\begin{equation*}
|\sum_{n \ioe x} (\mu(n)/n) \log(x/n)-1 | \ioe
\begin{cases}
1/9\,780\,919 \quad &( x \soe  2.5 \times 10^{12})\, , \\
(8.55 \times 10^{-6})/\log x \quad &( x \soe 2.5 \times 10^{11})\, , \\
0.162/ (\log x)^2 \quad &( x \soe 3) \,.
\end{cases} 
\end{equation*}
\end{thmintro}
\begin{thmintro} \label{thm C}
Posons $M(t)=\sum_{n \ioe t} \mu(n)$ et $m(t)=\sum_{n \ioe t} \mu(n)/n$, on a :
\begin{equation*}
\frac{2}{3} \ioe  \frac{\sup_{t \ioe x}|m(t)|t}{\sup_{t \ioe x }|M(t)|} \ioe \frac{3}{2} \quad ( x \soe 94) \;. 
\end{equation*}
\end{thmintro}
\begin{thmintro} \label{thm D}
\begin{equation*}
\varlimsup_{x\to \infty}  |\sum_{n \ioe x} \mu(n)/n|\sqrt{x} > \sqrt{2} \;.
\end{equation*}
\end{thmintro}

La présentation de nos résultats étant faite, venons-en à la partie historique puis aux méthodes. La fonction de Tchebichef $\psi$ fut l'une des premières fonctions arithmétiques encadrée efficacement.  Au début du 20\ieme\ siècle, on savait,  grâce aux travaux menés par Landau (depuis sa thèse en 1899) et par Axer, que les trois résultats suivants étaient équivalents :  
\begin{equation}
M(x)=\sum_{n \ioe x} \mu(n)=o(x), \quad 
xm(x)=x \sum_{n \ioe x} \frac{\mu(n)}{n}=o(x), \quad 
\psi(x)-x=o(x) \;. \label{Landau}
\end{equation}
Mais le sens précis du mot équivalent est à travailler dans chaque contexte, pour Landau cela signifiait sans doute ne pas utiliser d'analyse complexe pour passer d'un résultat à l'autre. Il y a pourtant des différences entre ces fonctions. En particulier,  les techniques développées pour évaluer explicitement chacune d'entre elles diffèrent largement.  

Ainsi on peut traiter $\psi(x)-x$ par les méthodes de l'analyse complexe en utilisant la formule explicite de von Mangoldt, une région sans zéro explicite et un calcul numérique des premiers zéros de $\zeta$. En revanche cette méthode n'a pas encore été appliquée avec le même succès pour traiter $M$ ou $m$ explicitement, elle aboutit pour l'instant à des majorations faisant intervenir des constantes bien trop grandes pour être utile en pratique. Pour avoir une idée de la taille des constantes obtenues par les formules de Perron pour $m$ voir l'article \cite{trudgian2015explicit}.   

 Une autre voie pour évaluer $M$ consiste  à utiliser des formes explicites de l'implication $\psi(x)-x =o(x) \Rightarrow M(x)=o(x)$. Schoenfeld dans l'article \cite{SchoenfeldMobius}  est le premier à obtenir des majorations explicites de $|M(x)/x|$ qui tendent vers $0$. Son travail est poursuivi par El Marraki \cite{MarrakiMobiusIII} puis par  Ramaré \cite{from}. Dans ces résultats qui partent d'identités de convolutions,  il s'agit tout d'abord d'utiliser une version explicite de $M(x)=O(x)$ la plus fine possible, pour El Marraki et Ramaré c'est l'inégalité $|M(x)| \ioe x/4345$ avec $x \soe  2\,160\,535$. La qualité de cette estimation pourtant moins forte que le théorème des nombres premiers se répercute sur les deux majorations asymptotiquement plus petites présentées ci-après, en outre sur l'intervalle $[10^{16}, \, 3\times 10^{24}]$ c'est la plus petite majoration de $|M(x)|$ dont on dispose. Voici certains des meilleurs résultats explicites obtenus à ce jour pour la fonction $|M(x)|$ et c'est à partir d'eux que nous avons obtenus les inégalités semblables sur $|m(x)|$ du théorème \ref{thm A}  présentées au début de l'introduction.
\begin{align}
\text{(Cohen-Dress-El Marraki, 1996)} \quad  : \quad   |M(x)| & \ioe \frac{x}{4345}  \quad (x \soe  2\,160\,535), \label{CDM}\\
\text{(Ramaré, 2013)} \quad  : \quad    |M(x)| & \ioe \frac{0.013 x}{\log x} \quad(x \soe 97\,067) \,,\label{donne pour M} \\
\text{(El Marraki, 1995)} \quad  : \quad    |M(x)| & \ioe  \frac{362.7 x}{(\log x)^2} \quad(x \soe 2)\,.  
\end{align}
La version explicite de $M(x)=O(x)$ est le fruit d'améliorations successives des idées élémentaires -- dans le sens où il n'est pas fait recours à l'analyse complexe -- développées par Tchebichef  pour obtenir les premiers encadrements de $\psi(x)$ en 1850. Von Sterneck arrive à prouver ainsi dans une publication de 1898 que $|M(x)| \ioe x/9+8$ pour tout $x\soe 1$.  Ce résultat a été amélioré  par MacLeod dans~\cite{Macmobius}, Costa Pereira dans~\cite{CostaPPsiM}, Dress, El Marraki dans
\cite{DressMobiusMarrakiII} et les trois co-auteurs Cohen, Dress, El Marraki dans~\cite{MobiusCohenDressMarraki}. Ces travaux reposent tous sur l'emploi d'une fonction de la forme $H(t)=1-\sum_{r}c_r \lfloor t/r \rfloor$ vérifiant $\sum_r c_r/r=0$ et qui approche bien la fonction définie sur $[1,\, +\infty[$ qui vaut constamment $1$. Nous allons voir par la suite que de telles fonctions jouent un rôle primordial dans notre travail.

Obtenir des encadrements explicites de la fonction sommatoire $m(x)$, ainsi que de la version lissée
$\check{m}(x)=\sum_{n \ioe x} (\mu(n)/n) \log(x/n)$, constitue un enjeu important de la 
théorie explicite des nombres premiers, comme en témoigne leur usage dans le travail  
d'Helfgott \cite{helfgott2012minor}  sur le problème de Goldbach ternaire. Voir également la récente
prépublication de Zuniga Alterman \cite{alterman2020logarithmic}.
Pour ce faire, on peut chercher à exploiter des identités plus ou moins sophistiquées reliant 
$m(x)$ à des intégrales pondérées de $M$. C'est la voie suivie par  El Marraki
dans \cite{preprint},  Ramaré \cite{RamarExplicitMob} en s'appuyant sur un travail 
de Balazard \cite{BalazarMobiusEnglish}.  Balazard introduit la fonction intermédiaire
\begin{equation}
m_1(x)=\sum_{n \ioe x}\mu(n) \left(\frac{1}{n}- \frac{1}{x}  \right)=m(x)-\frac{M(x)}{x} 
\end{equation} 
et par exemple il obtient l'identité
\begin{equation} 
m_1(x)= \frac{1}{x} \int_{1}^{x} M(x/t) \epsi_1'(t) \diff{t}+\frac{8}{3x}- \frac{4}{x^2} \Big(1-\frac{1}{3x^2} \Big) \quad (x\soe 1)\,,\label{bal2}
\end{equation} 
avec  la fonction dérivable $\epsi_1$ qui vérifie
\begin{equation}
\epsi_1(t) =\frac 13 - \frac 1{3t}
+ \frac 43 \frac{\{t\}^3 - \frac 32 \{t\}^2+\frac12 \{t\}}{t^2} -\frac 13 \frac{\{t\}^4-2\{t\}^3+\{t\}^2}{t^3} \quad \text{et}  \quad  0 \ioe \epsi_1'(t) \ioe \frac{1}{t^2}\;. \label{avant deriv}
\end{equation}
Dans cet article et la thèse \cite{daval2019identites} nous approfondissons ces idées en développant un cadre général en vue  d'établir des identités du type de \eqref{bal2}. La méthode de fabrication de ces identités différe de l'article \cite{BalazardHal}, il y a plus de souplesse et en particulier la fonction intégrée contre $M(x/t)$ n'est plus nécessairement une fonction  dérivée.
\begin{thm}\label{thm1}
1) Soit $g:[0,1]\to \CC$ une fonction intégrable vérifiant $\int_0^1 g(y)\diff{y}=1$. Pour tout $x \soe 1$ on a l'égalité 
\begin{equation*}
m_1(x)=m(x)-\frac{M(x)}{x} =  \int_1^x \frac{M(x/t)}{x/t} G(t)  \frac{\diff{t}}{t} + \frac{1}{x}\int_{1/x}^1 \frac{g(y)}{y}\diff{y}  \quad \text{où} \quad  G(t)=1 -\frac 1t\sum_{n\leqslant t} g\Big(\frac{n}{t}\Big)\,. \label{m1 par M}
\end{equation*}
2) Soit $h:[0,1]\to \CC$ une fonction intégrable vérifiant $\int_0^1 h(y)\diff{y}=0$. Pour tout $x \soe 1$ on a l'égalité 
\begin{equation*}
m_1(x)=m(x)-\frac{M(x)}{x} =\int_{1}^{x}m(x/t) H(t) \frac{\diff{t}}{t^2} -\int_0^{1/x} h(y) \diff{y} \quad \text{où} \quad  H(t)=1-\sum_{n \ioe t}h\Big(\frac{n}{t}\Big)\,. \label{m1 par m}
\end{equation*}  
\end{thm}
Pour exploiter au mieux du point de vue numérique ces résultats généraux et ainsi obtenir le théorème  \ref{thm A}, nous réutilisons avec profit la fonction $H$ exhibée par Cohen, Dress et El Marraki pour obtenir l'inégalité \eqref{CDM}.   
Puis en utilisant d'autres identités intégrales semblables à celles du théorème \ref{thm1} nous  exprimons  $\check{m}(x)$ à l'aide de $m_1(x)$ et nous obtenons le théorème  \ref{thm B}.

Par ailleurs, nos méthodes permettent également d'obtenir des encadrements généraux reportés dans le théorème \ref{thm C}, ces résultats montrent une certaine proximité entre $m$ et $M$  que la formule sommatoire d'Abel ne pourrait pas donner même qualitativement et qui ne semblent pas se trouver dans la littérature.   Depuis l'article de Odlyzko et te Riele  \cite{odlyzko1984disproof} on sait que la conjecture de Mertens qui énonçait que $|M(x)| \ioe \sqrt{x}$ pour tout $x\soe 1$ est fausse. La réfutation est en fait plus précise car elle montre  que l'inégalité est fausse une infinité de fois, le dernier record de ce type  est le résultat $\varlimsup_{x\to \infty}  |M(x)|x^{-1/2}>1.837625$ établi par Hurst \cite{hurst2018computations}. En 1897, lorsqu'il a publié sa conjecture, Mertens s'était notamment appuyé sur une vérification pour tous les $x$ inférieurs à $10\,000$. En faisant la même chose pour la fonction $m$ on serait conduit à conjecturer que $x|m(x)|  \ioe \sqrt{2x}$ pour tout $x\soe 1$.   En raffinant le théorème \ref{thm C} on aboutit à l'énoncé du théorème \ref{thm D} sans avoir à adapter la méthode de Odlyzko et te Riele à la fonction $m$.

\section{\texorpdfstring{De $M$ vers $m$}{De M vers m}}

\subsection{Outils pour convertir des inégalités explicites}

Le lemme suivant est tiré de l'identité de Balazard \eqref{bal2} mais les calculs d'intégrales \eqref{Muntz G1} sont nouveaux. On pourra convertir des résultats  $|M(x)|  \ioe A_j x^{\theta}/(\log x)^j$ vers le même type de majorations concernant $x|m_1(x)|$ puis $x|m(x)|$ et le théorème des nombres premiers correspond à $\theta=1$. 
\begin{lem} \label{machinerie epsi}
Il existe une fonction $G_1 : \RR_+ \rightarrow \RR_+$ telle que  pour tous $T,x,j,\theta  \in\RR$ tels que $1< T \ioe x$, $j \soe 0$ et $\theta>-1$, on a :
\begin{equation}
|m_1(x)| \ioe \sup_{T < u < x} \big( u^{-\theta} |M(u)| \log^j u \big) \times \int_{1}^{\infty} G_1(t) t^{-\theta +\frac{j}{\log T}}\diff{t} \times \frac{x^{\theta-1}}{\log^j x}  +  \frac{R_T(x)}{x}\label{boite G1}
\end{equation}
avec $R_T(x)= 8/3+  (1/x) \int_{1}^{T} |M(t)|   \diff{t}$. \newline

 Pour tout $s \in\CC$ tel que $\Re s>-1$, $s\neq 1$, on a :
\begin{equation}
\int_{1}^{\infty} G_1(t) t^{-s}\diff{t}
= \frac{1}{s-1}- \frac{8\zeta(s)}{(s+1)(s+3)}\;.  \quad \text{Et aussi }  \int_{1}^{\infty} G_1(t) t^{-1}\diff{t}=\frac{3}{4}-\gamma\;.  \label{Muntz G1}
\end{equation}
\end{lem}
 On retrouve une inégalité de l'article \cite[proposition 7 p.~9]{BalazardHal} de Balazard  en prenant $T=x$ .
 
 Pour ne pas ralentir la lecture, dans la démonstration suivante nous utiliserons le lemme \ref{log puissance k} qui sera prouvé un peu après. Quand $j=0$ il n'y pas besoin de ce lemme et pour suivre le principe de notre démarche c'est suffisant. De même la démonstration du théorème \ref{thm1} est renvoyée en fin de section.
\begin{proof}[Démonstration  du lemme \ref{machinerie epsi}]
 Par le théorème \ref{thm1} on a pour tout $x \soe 1$  l'égalité
\begin{equation}
m_1(x) =  \int_1^x \frac{M(x/t)}{x/t} G(t)  \frac{\diff{t}}{t} + \frac{1}{x}\int_{1/x}^1 \frac{g(y)}{y}\diff{y}  \quad \text{ où } \quad  G(t)=1 -\frac 1t\sum_{n\leqslant t} g\Big(\frac{n}{t}\Big) \,. \label{U Michel} 
\end{equation}
Dans l'article \cite{BalazardHal} Balazard introduit une fonction qu'il note $\epsi'_1$, dans notre thèse \cite[proposition 8 p.~19]{daval2019identites} nous montrons que  $\epsi'_1=G_1$ pour $g_1(y)=4y(1-y^2)$ dans les égalités \eqref{U Michel}. Dans \cite[propositions 2 p.~7 et 6 p.~9]{BalazardHal}  l'encadrement suivant est prouvé : 
\begin{equation}
0 \ioe \epsi'_1(t)= G_1(t)= 1 -\frac 1t\sum_{n\leqslant t} g_1\Big(\frac{n}{t}\Big) \ioe \frac{1}{t^2}\,. \label{encadrement G1}
\end{equation}
Puisque la fonction  $g_1$ est positive on a $0 \ioe \int_{1/x}^1 g_1(y)/y\diff{y} \ioe \int_{0}^1 g_1(y)/y\diff{y}=8/3$ et ainsi
\begin{equation}
|m_1(x)|  \ioe  \int_1^{x/T} \frac{|M(x/t)|}{x/t} |G_1(t)|  \frac{\diff{t}}{t}+\int_{x/T}^{x} \frac{|M(x/t)|}{x/t} |G_1(t)| \frac{\diff{t}}{t} + \frac{8/3}{x} \,.  \label{decoupe G1}
\end{equation}
Pour la deuxième intégrale dans \eqref{decoupe G1}, on a d'après la ligne \eqref{encadrement G1} la majoration 
\begin{equation}
\int_{x/T}^{x} \frac{|M(x/t)|}{x/t} |G_1(t)| \frac{\diff{t}}{t} \ioe  \int_{1}^{x/T} \frac{|M(x/t)|}{x/t}  \frac{1}{t^2} \frac{\diff{t}}{t} \underset{u=x/t}{=}\frac{1}{x^2}\int_{1}^{T} |M(u)|  \diff{u}\,. \label{G1 entre x/T et x}
\end{equation}
Pour la première intégrale dans \eqref{decoupe G1}, posons $\alpha=1-\theta$ pour alléger un peu l'écriture, on a
\begin{align}
\int_{1}^{x/T} \frac{|M(x/t)|}{x/t} &|G_1(t)| \frac{\diff{t}}{t} \ioe \int_{1}^{x/T} \frac{\sup_{1<v <x/T}|(x/v)^{\alpha} \log^j(x/v) M(x/v)/(x/v)|}{(x/t)^{\alpha} \log^j(x/t)} |G_1(t)| \frac{\diff{t}}{t} \nonumber \\
 &\underset{u=x/v}{=}\sup_{T<u <x}|u^{\alpha-1} \log^j(u) M(u)|  \int_{1}^{x/T} \frac{1}{ \log^j(x/t)} |G_1(t)|t^{\alpha} \frac{\diff{t}}{t} \frac{1}{x^{\alpha}}  \label{G1 entre 1 et x/T}\,.
\end{align} Par le lemme \ref{log puissance k} appliqué avec la fonction $F(t)=|G_1(t)|t^{\alpha} $  on obtient
\begin{equation} \int_{1}^{x/T} \frac{1}{ \log^j(x/t)} |G_1(t)|t^{\alpha} \frac{\diff{t}}{t} 
\ioe  \frac{1}{ \log^j(x)} \int_{1}^{x/T}  |G_1(t)|t^{-1+\alpha+\frac{j}{ \log T}} \diff{t} \label{G1 et log}  
\end{equation} 
ce qui, ramené dans \eqref{G1 entre 1 et x/T} en sachant que $G_1$ est  une fonction positive et ajouté à \eqref{G1 entre x/T et x} dans l'inégalité \eqref{decoupe G1}, termine la preuve de la majoration \eqref{boite G1}.

L'égalité \eqref{Muntz G1} est le prolongement analytique de la formule de Müntz, appliquons le  paragraphe II.11 du livre de Titchmarsh \cite{titchmarsh1986theory} avec $F(y)=g_1(y) \mathrm{1}_{[0,1]}(y)$ et les conditions sont vérifiées car  $g'_1$ est continue et on a $g(1)= g'(1)=0$, (voir notre thèse \cite[p.~10 à p.~16]{daval2019identites}  pour plus de détails sur la formule de Müntz dans ce contexte où le support est fini). On obtient
\begin{equation}
\int_{0}^{\infty} G_1(t) t^{-s}\diff{t}
= -\zeta(s) \int_{0}^{1}g_1(y) y^{s-1} \diff{y}=- \frac{8\zeta(s)}{(s+1)(s+3)} \quad(0< \Re(s) <1) \label{part1}
\end{equation}
puis l'égalité \eqref{Muntz G1}  pour $0<\Re(s) <1$ en soustrayant  $\int_{0}^{1} G_1(t) t^{-s}\diff{t}
=\int_{0}^{1} t^{-s}\diff{t}=  -1/(s-1)$   à \eqref{part1}. D'après les inégalités \eqref{encadrement G1}, $s \mapsto \int_{1}^{\infty} G_1(t) t^{-s}\diff{t}$ est holomorphe pour $\Re s>-1$ donc par prolongement analytique pour  $\Re(s) >-1$  en supposant $s\neq 1$,  on obtient la première égalité dans la ligne \eqref{Muntz G1}. 
Pour la deuxième égalité dans \eqref{Muntz G1}, compte tenu de $\zeta(s)=1/(s-1)+\gamma+O(|s-1|)$, on a:
\begin{align*}
&\frac{1}{s-1} \int_{0}^{1}g_1(y)\diff{y}-\zeta(s)
\int_{0}^{1}g_1(y)y^{s-1}\diff{y}\\
=-&\int_{0}^{1}g_1(y) \frac{y^{s-1}-1}{s-1}\diff{y}-\gamma \int_{0}^{1}g_1(y)\diff{y}+O(|s-1|) \\
\underset{ s \to 1}{\longrightarrow} - &\int_{0}^{1}g_1(y) \log y \diff{y}
-\gamma \int_{0}^{1}g_1(y)\diff{y}\;.
\end{align*} 
On finit en calculant par une primitive $\int_{0}^{1}g_1(y) \log y \diff{y}$ pour $g_1(y)=4y(1-y^2)$.\end{proof}

\begin{rem} 
L'article \cite{BalazardHal} ne procède pas comme nous pour obtenir la deuxième formule de la ligne \eqref{Muntz G1} et Balazard pointe l'utilité de la positivité pour le calcul de l'intégrale. Il y a une petite erreur de calcul p. 10, il est écrit (transcrit avec nos notations) que $\int_{1}^{\infty} G_1(t) t^{-1}\diff{t}=271/360-\gamma \simeq 0.1755$, c'est à comparer à la vraie valeur $3/4-\gamma\simeq 0.1727$. Ce nombre est le meilleur facteur multiplicatif que l'on puisse obtenir en passant de $|M(x)/x|$ à $|m_1(x)|$ par le lemme \ref{machinerie epsi} au sujet du théorème des nombres premiers. 
\end{rem}
Nous utiliserons le lemme suivant pour convertir les premiers résultats sur $|m|$ du type $|m(x)|  \ioe A'_j x^{\theta-1}/(\log x)^j$ obtenus par le lemme \ref{machinerie epsi} en des résultats sur $|m_1|$. Le but est de revenir alors à des inégalités sur $|m|$ puis d'itérer cette procédure. Nous prenons ce chemin moins direct car la fonction $H_2$ du lemme ci-dessus est plus efficace  pour nos conversions (le meilleur facteur multiplicatif que l'on puisse obtenir  en passant de $|m(x)|$ à $|m_1(x)|$ est cette fois $\simeq 0.00038$). Remarquons également que le supremum porte sur un intervalle  plus petit que dans le lemme \ref{machinerie epsi}. 
 \begin{lem} \label{machinerie H2 CDM}
Soit $K=100\,822$. Il existe une fonction $H_2 : \RR_+ \rightarrow \RR$ telle que pour tous $T,x,j,\theta \in\RR$ tels que  $1< T \ioe x$, $x \soe 5\times 10^{13}$, $j \soe 0$ et $\theta>0$, on a :
\begin{equation*}|m_1(x)| \ioe \sup_{T < u < x/K} \big(u^{-\theta} u|m(u)| \log^j u \big) \times \int_{1}^{\infty}|H_2(t)|  t^{-1-\theta +\frac{j}{\log T}}\diff{t} \times \frac{x^{\theta-1}}{\log^j x}\\
 +  \frac{R_T}{x} 
 \end{equation*}
avec $R_T=22527.5   \int_{1}^{T} |m(t)|  \diff{t} +6$. \newline

Pour tout $\delta \in\RR$ tel que  $0<\delta <1$, on a :
\begin{equation}
\int_{1}^{\infty}  |H_2(t)| t^{-2+\delta} \diff{t} \ioe  \frac{1}{1-\delta}  \left( \frac{6/\pi^2 \times 4345 \times 22527.5}{\delta }  \right)^{\delta}  \times 
\dfrac{\pi^2}{6}\dfrac{1}{4345} \;.    \label{lemme integre}
\end{equation}
\[
\text{Et on a } \int_{1}^{\infty}  |H_2(t)| t^{-2} \diff{t} \ioe \dfrac{\pi^2}{6}\dfrac{1}{4345} \;.
\]
\end{lem}

Avant d'en faire la démonstration, voyons quelques utilisations numériques de ce lemme. 
Les inégalités suivantes  explicitent l'implication 
$m(x) \ll_{\delta} x^{-\delta} \Rightarrow  M(x)/x \ll_{\delta} x^{-\delta}$, on a :
\begin{equation*}
|m_1(x)| \ioe \sup_{1 < u < x/10^{5}} |m(u)| u^{\delta}   \times \frac{C_{\delta}}{x^\delta}
 +  \frac{6}{x}  \quad (x \soe 5 \times 10^{13} )
\end{equation*}
où la constante $C_{\delta}$ est donnée par la table suivante. 
\[
\begin{array}{|c|c|c|c|c|c|c|} \hline
\delta      & 0.5     & 0.1        & 0.05                & 0.01       &0.001     & 0^-  \\ \hline
 C_{\delta}  & 8.261  & 0.0032     & 0.00114             & 0.000479    &0.000389 & 0.000378 \\  \hline 
  T_{\delta}          & 8      & 22030      & 5.2 \times 10^{21}   & 2.7 \times 10^{43} & 2\times 10^{434} & +\infty\\  \hline
\end{array}
\]
Le cas $\delta=0.5$ semble plus mauvais qu'utiliser la formule sommatoire d'Abel  mais l'intervalle est ici plus court (on obtiendrait $2$ au lieu de $8.261$, voir également la proposition \ref{prop:conversion-m-vers-M-racince} où l'on obtient $0.3$). 

De même explicitons l'implication $m(x) \ll 1/\log x  \Rightarrow  M(x)/x \ll 1/\log x $, avec les rangs  $T_{\delta}$ ci-dessus on a :
\begin{equation*}
|m_1(x)| \ioe \sup_{T_{\delta} < u <  x/10^{5}} |m(u)| \log u  \times \frac{C_{\delta}}{\log x}
 +  \frac{22600   \int_{1}^{T_{\delta}} |m(t)|  \diff{t} }{x}  \quad \big(x \soe \max( 5 \times 10^{13} , T_{\delta}) \big) \,.
 \end{equation*}

C'est dans l'article \cite{MobiusCohenDressMarraki} que nous trouvons la fonction $H_2$ du lemme \ref{machinerie H2 CDM} où elle est utilisée par les auteurs dans la démonstration de $|M(x)|\ioe x/4345$ et nous notons ci-dessous les propriétés de cette fonction qui nous seront utiles.
\begin{prop} \label{FDress}
Il existe une fonction $H_2$ définie pour $t\soe 1$ qui s'écrit $H_2(t)=\sum_{r} c_r \lfloor t/r \rfloor$ avec un nombre fini
de nombres réels $c_r$ avec $r \soe 1$ et telle que :
\begin{enumerate}
\item $\max r \ioe 5 \times 10^{13}$ 
\item  $ |H_2| \ioe 22\, 527.5  $ 
\item $\sum_{r} \dfrac{c_r}{r} =0$ 
\item $ |\sum_{r} c_r | \ioe 6$ 
\item $\int_{1}^{\infty} |H_2(t)|t^{-2} \diff{t} \ioe \dfrac{\pi^2}{6}\dfrac{1}{4345}$
\item  $H_2(t)=0$ pour    $ t \ioe K=100\,882$.
\end{enumerate} 
\end{prop}
\begin{proof} Tous ces points se trouvent dans \cite{MobiusCohenDressMarraki}. 
\begin{enumerate}
\item En bas de la page 58 on trouve que $\max_r r=47\,666\,734\,237\,381.39$.
\item Voir le tableau page 59 en sachant que les auteurs notent $G=|H_2|$. Cette majoration provient de la méthode de Costa Pereira de \cite{CostaPPsiM}.
\item C'est le point (iv) de la page 58, il entraîne que la fonction $H_2$ est périodique. 
\item La troisième ligne de la page 59 indique que $|1+\frac{1}{2}\sum_{r} c_r |=2$. 
\item La méthode ne peut pas dépasser $J= \frac{6}{\pi^2}\int |H_2(u)| u^{-2} \diff{u}$ (voir remarque page 109 de l'article de Dress 
et El Marraki~\cite{DressMobiusMarrakiII}).
\item Page 58 c'est l'entier noté $k_0$. 
\end{enumerate} \end{proof} 
\begin{rem}
Pour le point 5 de la preuve ci-dessus, il est à noter que les auteurs indiquent après le tableau page 61 que l'intégrale  vérifie $J=\frac{6}{\pi^2}\int |H_2(u)| u^{-2} \diff{u} \approx 1/4930$ mais puisque ce n'est pas utilisé dans l'article \cite{MobiusCohenDressMarraki} on ne peut pas affirmer que l'erreur dans ce calcul est contrôlée. Les coefficients $c_r$ de la fonction $H_2$ ne sont pas tous donnés dans l'article donc il ne nous est pas possible de refaire le calcul de cette intégrale. La valeur que l'on peut utiliser avec certitude est $1/4345$ mais on peut  espérer encore gagner un facteur d'environ $0.88$ sur nos résultats pour $m_1$ dans les théorèmes \ref{m petit 4343}, \ref{propo 1 sur log } et \ref{m1 log carre 119}.
\end{rem}
Le lemme suivant permet de passer des parties entières de la fonction $H_2$ décrite dans la proposition \ref{FDress} aux sommes définissant les fonctions $H$ dans le théorème \ref{thm1}.
\begin{lem} \label{recyclage}
 Soit $(c_r)$ une suite finie de nombres complexes vérifiant  $\sum_{r} c_r/r=0$.
On pose $h(y)=\sum_{r}c_r\mathrm{1}_{[0,1]}(ry)$. Alors on a  $\int_{0}^{1}h(y) \diff{y}=1$, 
\begin{equation*}
\sum_{n \ioe t}h\left(\frac{n}{t}\right)=\sum_{r}c_r \lfloor t/r \rfloor=-\sum_{r}c_r \{ t/r \} \quad  \text{ et } \quad \int_0^{1/x} h(y) \diff{y}=-\frac{\sum_{r}c_r}{x} \quad (x \soe \max r)\,. 
\end{equation*} 
\end{lem}
\begin{proof}
On a $\sum_{n \ioe t} \mathrm{1}_{[0,1]}(rn/t)=\sum_{rn \ioe t}1=\lfloor t/r \rfloor$ et on obtient la première formule par linéarité. 
La fonction $t^{-1}\sum_{r}c_r \lfloor t/r \rfloor$ tend d'une part vers $\sum_{r} c_r /r$ qui vaut $0$ par hypothèse et d'autre part vers $\int_{0}^{1}h(y) \diff{y}$. Pour une seule fonction indicatrice $h(y)=\mathrm{1}_{[0,1]}(ry)$, pour   $x \soe r$ on a $h(y)=1$ si $y \in [0,1/x]$ et donc 
\[
\int_{0}^1 h(y) \diff{y}-\int_0^{1/x} h(y) \diff{y} =  \frac{1}{r}-\frac{1}{x}  \quad (x \soe r)\,.  
\]
Par linéarité, on obtient la deuxième formule à condition que $x \soe \max r$. 
\end{proof}

\begin{proof}[Démonstration  du lemme \ref{machinerie H2 CDM}]
 Soit fonction $H_2(t)=1-\sum_{r}c_r \lfloor t/r \rfloor$ provenant de l'article \cite{MobiusCohenDressMarraki} décrite dans la proposition \ref{FDress} de notre article. Par l'identité du théorème \ref{thm1} et le lemme \ref{recyclage} on a
\begin{equation}
m_1(x) =\int_{1}^{x}m(x/t) H_2(t) \frac{\diff{t}}{t^2} + \frac{\sum_{r}c_r}{x} \quad (x\soe  \max r) \;. \label{identite H_2}
\end{equation}
D'après  la proposition \ref{FDress} on a  $|\sum_{r}c_r| \ioe 6 $,  $H_2(t)=1$ pour tout $t \ioe K$ et $\max r \ioe 5 \times 10^{13}$. Ainsi en supposant que $K  \ioe x/T$ l'égalité \eqref{identite H_2} donne
\begin{equation}
|m_1(x)|  \ioe \int_{K}^{x/T} |m(x/t) H_2(t)| \frac{\diff{t}}{t^2}+\int_{x/T}^{x} |m(x/t) H_2(t)| \frac{\diff{t}}{t^2} + \frac{6}{x} \quad (x\soe  \max 5 \times 10^{13}) \;.  \label{deux integrales}
\end{equation}
Commençons par majorer la deuxième intégrale dans \eqref{deux integrales}, on a
\begin{equation}
\int_{x/T}^{x} |m(x/t) H_2(t)| \frac{\diff{t}}{t^2} \ioe \| H_2 \|_{\infty}\int_{x/T}^{x} |m(x/t)| \frac{\diff{t}}{t^2}\underset{u=x/t}{=} 
\frac{\| H_2 \|_{\infty}}{x} \int_{1}^{T} |m(u)|  \diff{u}\;.  \label{morceau K x T} 
 \end{equation}
Posons $\alpha=1- \theta$. Par le lemme \ref{log puissance k} la majoration de la première intégrale dans \eqref{deux integrales} donnera 
\begin{equation} \int_{K}^{x/T} |m(x/t) H_2(t)| \frac{\diff{t}}{t^2} \ioe
\frac{\sup_{T<u <x/K}|u^{\alpha} \log^j(u) m(u)|}{\log^j(x)} \int_{1}^{\infty} |H_2(t)| t^{-2+\alpha+\frac{j}{\log T}} \diff{t}\;,  
\end{equation}
 la preuve est identique aux lignes \eqref{G1 entre 1 et x/T} et \eqref{G1 et log}  de la démonstration du lemme \ref{machinerie epsi} sauf que les valeurs de $m(x/t)$ entre $t=x/K$ et $t=x$ n'interviennent pas.
 
Pour compléter la preuve de la première partie de l'énoncé traitons le cas $K  > x/T$, $|m_1(x)|-6/x$ sera alors inférieur à l'intégrale du membre de droite de l'inégalité \eqref{morceau K x T} puisque déjà inférieur ou égal à la même intégrale sur l'intervalle plus court $[K,\,x]$.

Passons à la deuxième partie de l'énoncé. Posons  $I=\int_{1}^{\infty}  |H_2(t)| t^{-2} \diff{t} \neq 0$ et découpons  en deux parties l'intervalle d'intégration de l'intégrale $\int_{1}^{\infty}  |H_2(t)| t^{-2+\delta} \diff{t}$ qui est convergente par comparaison. Soit $B \soe 1$, on a : 
\begin{equation*}
\int_{B}^{\infty} |H_2(t)| t^{-2+\delta} \diff{t} \ioe \frac{\| H_2 \|_{\infty}}{1-\delta} \frac{1}{B^{1-\delta}}  \quad  \text{ et }  \quad \int_{1}^{B} |H_2(t)| t^{-2 +\delta} \diff{t}  \ioe I B^{\delta}\;. 
\end{equation*}
On choisit $B= \dfrac{\| H_2 \|_{\infty}}{I \delta}$ (pour minimiser la somme), on obtient :
\begin{equation*}
\int_{1}^{\infty} |H_2(t)| t^{-2+\delta} \diff{t} \ioe \frac{\| H_2 \|_{\infty}}{1-\delta} \frac{1}{B^{1-\delta}}+ I B^{\delta} = \frac{\delta I}{1-\delta}B^{\delta}+IB^{\delta}= \frac{1}{1-\delta} IB^{\delta}\,.
\end{equation*}
Par la proposition \ref{FDress} on a $I\ioe (\pi^2/6)/4345$ et $\| H_2 \|_{\infty}\ioe 22\,527.5$ ce qui termine la démonstration de la majoration \eqref{lemme integre}. 
\end{proof}

Dans les majorations de $|m_1(x)|$ issues des identités du théorème \ref{thm1} on rencontre  par exemple pour les formules passant par $M$ les intégrales
\begin{equation}
\int_1^x \frac{|M(x/t)|}{x/t} |G(t)|  \frac{\diff{t}}{t}
 \quad \text{ puis } \quad \int_1^{x/T} \frac{1}{ \log^j(x/t)} |G(t)|  \frac{\diff{t}}{t} \overset{?}{ \ioe } \frac{A(T,j)}{\log^j(x)} \label{question}
\end{equation}
où la deuxième intégrale arrive quand on  convertit une version explicite de $|M(u)/u| \ll 1/\log^j u$. C'est un morceau plus difficile à traiter directement, le point d'interrogation indique que c'est l'ordre de grandeur que l'on souhaite atteindre et la question est d'obtenir une bonne constante explicite. Pour $j=0$ la constante $A(T,j)=\int_{1}^{\infty}|G(t)|t^{-1}\diff{t}$ convient évidemment 
 et pour $j>0$ on aimerait dans l'idéal avoir cette constante. Le lemme suivant permet presque d'y arriver. Il servira de la même manière  pour les identités du théorème \ref{thm1} qui passe par $m$ et les fonctions $H$ et donc la notation $F(t)$ désignera principalement $|G(t)|$ ou $|H(t)|t^{-1}$.
\begin{lem} \label{log puissance k}
Soit $F : [1,\, +\infty[  \rightarrow \RR_+$ une fonction positive, $T > 1$ un nombre réel, $j>0$ un nombre réel. Pour tout $x\soe T$, on a :
\begin{equation*} 
\frac{1}{\log^j x }\int_{1}^{x/T} F(t) t^{-1} \diff{t} \ioe \int_{1}^{x/T} \frac{1}{\log^j(x/t)} F(t) t^{-1} \diff{t} 
\ioe  \frac{1}{ \log^j x }\int_{1}^{x/T} F(t)  t^{-1+\tfrac{j}{\log T}}  \diff{t} \;. 
\end{equation*} 
\end{lem}

\begin{proof}
Commençons par remarquer que pour  $r \soe j/\log T$  la fonction   $L_j(x)=x^r/\log^j x$ est croissante sur l'intervalle $[T,\, +\infty[$. En effet on a 
la factorisation 
\begin{equation*}
L_j'(x)=\frac{x^{r-1}}{ (\log x)^{j+1}}(r \log x- j)\;,
\end{equation*} 
c'est une fonction positive sur $[T,\, +\infty[$ quand $r \soe j/\log T$.
Posons $r = j/\log T$, on a 
\begin{equation} 
\int_{1}^{x/T}\frac{1}{\log^j(x/t)} \frac{F(t)}{t} \diff{t} =  \frac{1}{x^r} \int_{1}^{x/T} \frac{(x/t)^r}{\log^j(x/t)}\frac{F(t)}{t^{1-r}} \diff{t}
\ioe  \frac{1}{x^r} \frac{(x/1)^r}{\log^j(x/1)} \int_{1}^{x/T} \frac{F(t)}{t^{1-r}} \diff{t}\;.\label{K=1}
\end{equation}
Ce qui donne bien la majoration par $(\log x)^{-j} \int_{1}^{x/T} F(t)t^{-1+j/\log T} \diff{t}$.

La minoration de l'énoncé s'obtient directement en utilisant la croissance de la fonction 
$t \mapsto 1/ \log^j(x/t)$ pour $t \in [1,\,x/T[$ et la positivité de $F$.
\end{proof}
Pour $F$ positive et non identiquement nulle, d'après le lemme précédent avec $j=1$ on a  $\int_{1}^{x/T}  [1/\log(x/t)] F(t) t^{-1} \diff{t} \sim \int_{1}^{\infty} F(t)  t^{-1}   \diff{t} /  \log x $ quand $x \to \infty$ et pour envisager la suite du développement asymptotique on montre de la même manière que
\begin{equation*} 
 \int_{1}^{x/T} \frac{1}{\log(x/t)} F(t) t^{-1}\diff{t} \ioe \frac{1}{\log x }\int_{1}^{\infty} F(t)  t^{-1}   \diff{t}  
+ \frac{1}{ \log^2 x }\int_{1}^{\infty} F(t)  t^{-1+\tfrac{1}{\log T}}  \log t
 \diff{t}\;. 
\end{equation*}

\subsection{Intervalles bornés et comparaisons avec la fonction racine carrée}
Commençons par quelques remarques sur les vérifications informatiques. Dans l'article~\cite{hurst2018computations} Hurst utilise des algorithmes pour faire le calcul des nombres entiers $M(n)$ pour tous les  $n\ioe10^{16}$, l'auteur indique qu'il lui faut 1.35 jours pour faire les calculs de $M(n)$ jusqu'à $10^{14}$ et 7.5 mois  pour  les mener jusqu'à $10^{16}$. Dans des articles autour d'une étude plus théorique de la fonction de M\"obius, Ramaré pour ses besoins va couramment jusqu'à $10^{12}$ et Helfgott jusqu'à  $10^{14}$ (avec pour ce dernier de l'arithmétique d'intervalles). En réutilisant leurs calculs et nos méthodes de conversions nous n'auront pas à aller si loin (seulement $5\times10^6$). 

Nous souhaitons dans tout cet article faire moins de calculs sur $m$ que sur $M$, autrement dit ne pas refaire pour $m$ tout le travail qui a été fait sur $M$ et obtenir néanmoins des résultats du même ordre. Pour bien délimiter les ressources utilisées, nous isolons les vérifications numériques que nous avons eu à faire des résultats que nous avons convertis à partir de ceux obtenus pour $|M|$ ou $|m|$ par d'autres auteurs.  
\begin{lem}\label{calculs GP}
On a les vérifications informatiques suivantes :
\begin{equation*}
4343|m(x)| \ioe 1   \quad  x \in [2\, 160\, 605,\,5 \times 10^6[   \text{, }
 \log x |m(x)| \ioe 0.0130073   \quad x\in [97\,063,\,230\, 000[ \,.  
\end{equation*}
\end{lem}

\begin{proof}
La fonction  $m$ étant constante par morceaux il suffit de vérifier la première inégalité sur $[n,n+1[$ pour tous les nombres entiers $n$ considérés. Pour la deuxième inégalité si l'on vérifie $\log (n+1) |m(n)| \ioe b$ alors on prouve que $\log x |m(x)| \ioe b$ pour tout $x \in[n,n+1[$. Nous utilisons le logiciel PARI/GP avec une précision de $10^{-12}$ pour vérifier les encadrements :
\begin{align*}
4343|m(n)|& \ioe 1   \quad  n \in [2\, 160\, 605,\,5 \times 10^6-1] \\
\log (n+1) |m(n)|& \ioe 0.0130073   \quad n\in [97\,063,\,230\, 000-1] \,.  
\end{align*} 
Cela prend quelques minutes et termine la démonstration.
\end{proof} 
Dans notre article, pour les calculs concernant $\mu(n)$ pour $n \ioe 7000$ nous indiquerons dans les démonstrations \og{}par une vérification directe \fg{}. 
\begin{prop}[Dress/Hurst/Helfgott] \label{racine autres}
On a les estimations suivantes :
\begin{equation}
|M(x)|\ioe 0.5 \sqrt{x}  \quad  (201 \ioe x \ioe 7.7 \times 10^{9})\, , \quad |M(x)|\ioe 0.571 \sqrt{x}  \quad  (33 \ioe x \ioe 10^{16})\,,  \label{racineMHurst}
\end{equation}
\begin{equation}
x|m(x)|\ioe 0.5 \sqrt{x}  \quad  (3 \ioe x \ioe 7.7 \times 10^{9})\,.  \label{racinemHelfgott}
\end{equation} 
\end{prop}
\begin{proof}
Dress et Cohen ont prouvé que $|M(x)| \ioe 0.5 \sqrt{x}$ pour tout $x \in [201, T_M]$ pour $T_M=7\, 725 \, 038 \, 628$, le résultat est repris en détails dans \cite{DressMobiusI}. La deuxième majoration provient de l'article de Hurst \cite{hurst2018computations}.

Dans  \cite[eq. 2.11, p.~7]{helfgott2012minor} Helfgott établit que  $x|m(x)| \ioe 0.5\sqrt{x}$  pour tout $ x \in [3, T_m]$ pour $T_m=7  \,727\,068 \,587$.
\end{proof}
\begin{cor}\label{modeleracine}
On a les quatre estimations suivantes :
\begin{align*}
x|m_1(x)| &\ioe\begin{cases}
0.114 \sqrt{x}  \quad  (5 \times 10^6 \ioe x \ioe 7.7 \times 10^{9})\, , \\
0.129\sqrt{x} \quad   (7.7 \times 10^{9} \ioe x \ioe 10^{16})\,,   \\
 5.792\sqrt{x} \quad ( 10^{16} \ioe x \ioe 10^{21})  \,,  \end{cases} \\
x|m(x)| &\ioe \phantom{0.}0.701\sqrt{x} \quad ( 7.7 \times 10^{9} \ioe x \ioe 10^{16}) \, .
\end{align*}
\end{cor}

\begin{proof} Par le lemme \ref{machinerie epsi} avec les paramètres $\theta=0.5$, $j=0$ et $s=0.5$ on a
\begin{equation}
|m_1(x)| \ioe  \sup_{T<x < x}|M(u)|u^{-1/2} \times  \frac{b}{\sqrt{x}} + \frac{8/3}{x} +\frac{\int_{1}^{T} |M(t)| \diff{t}}{x^2} \quad (x \soe T) \label{moule1}
\end{equation}
avec $b= -2-(32/21)\zeta(1/2) \simeq 0.225$. La première estimation de la ligne \eqref{racineMHurst} peut s'écrire $\sup_{201<u<x}|M(u)|u^{-1/2}\ioe 0.5  $ pour tout $x$ vérifiant $ 201 \ioe x \ioe 7.7 \times 10^{9}$. Prenons donc $T=201$, un calcul direct fournit $\int_{1}^{201} |M(u)| \diff{u} =\sum_{n=1}^{200}|M(n)|=461$, ainsi l'inégalité \eqref{moule1} donne
\begin{equation}
 |m_1(x)| \ioe \frac{0.112652}{\sqrt{x}}+\frac{8/3}{x}+\frac{461}{x^2}  \quad(201 \ioe x \ioe 7.7 \times 10^{9}) \label{modele du pauvre} 
\end{equation} 
puis par décroissance de la fonction obtenue après multiplication par $\sqrt 
x$, on a  
\[
f(x)=0.112652  +(8/3)x^{-0.5} + 461x^{-1.5} \ioe f(5 \times 10^{6}) \ioe  0.114  \quad (x \soe 5 \times 10^{6}) \,. 
\]
Donc par comparaison avec \eqref{modele du pauvre} on obtient $\sqrt{x} |m_1(x)|\ioe 0.114 $ pour le même intervalle, ce qui termine la démonstration de la première estimation.

On part de la deuxième estimation dans \eqref{racineMHurst}  puis on applique la même démarche. On utilise la ligne \eqref{moule1} avec $T=33$ et par un calcul direct on a $\sum_{n=1}^{32}|M(n)|=59$. La fonction majorant $\sqrt{x}|m_1(x)|$ est cette fois $f(x)=0.12865+(8/3)x^{-0.5} + 59x^{-1.5}$, par décroissance elle est inférieure à $f(7.7\times 10^{9})<0.129 $ pour $ x \soe 7.7\times 10^{9}$. Ce qui termine la démonstration de la deuxième estimation. 

On utilise le lemme \ref{machinerie H2 CDM} avec les paramètres $\theta=0.5$, $T=3$, $j=0$ et $\delta=0.5$, on a
\begin{equation*}
|m_1(x)| \ioe \sup_{3 < u < x/K} \sqrt{u}|m(u)|  \times  \frac{8.261}{\sqrt{x}}
 +  \frac{22527.5 \times 1.5+6}{x}   \quad (x\soe 5 \times 10^{13})
\end{equation*}
où $K=100\,882> 10^5$. L'estimation \eqref{racinemHelfgott} donne $\sqrt{u}|m(u)| \ioe 0.5$ pour $u$ entre $3$ et $7.7 \times 10^9$, entre $7.7 \times 10^9$ et $10^{16}$ on utilise la deuxième estimation du lemme \ref{modeleracine} $\sqrt{u}|m(u)| \ioe 0.701$,  ainsi pour tout  $x\ioe K \times 10^{16}$ on a $\sup_{3 < u < x/K} \sqrt{u}|m(u)|   \ioe 0.701 $. On conclut puisque la fonction décroissante $f(x)= 0.701 \times  8.261 + (22527.5 \times 1.5+6)x^{-0.5}$ vérifie $f( 10^{16})<5.792$. 

Pour finir, on a $m_1(x)=m(x)-M(x)/x$ donc  $|m(x)|\ioe |M(x)/x|+|m_1(x)|$, on conclut en additionnant  les majorations pour $|M|$  et pour $|m_1|$ toutes deux valables  pour $x$ dans $[7.7\times 10^{9},\, 10^{16}]$. 
\end{proof}

Pour un bon programmeur qui aurait du temps les estimations du corollaire \ref{modeleracine} seraient surpassées pour  $x\ioe 10^{14}$ mais elles illustrent  bien le principe des conversions, ici cela permet de recycler rapidement des calculs fiables. De plus l'intervalle $[10^{14},\, 10^{16}]$ est réservé aux meilleurs programmeurs et machines.
Le dernier résultat sur $m_1$ est de nature différente  car son domaine de validité dépasse celui dont on est parti pour $M$ : $x \ioe 10^{16}$. Bien que la constante devant la racine semble grande, il sera meilleur que nos  estimations asymptotiques pour $m_1$ obtenues par nos autres méthodes  sur l'intervalle $[10^{16},\, 10^{21}]$.

\subsection{\texorpdfstring{Conversion d'une borne $|M(x)/x| \ioe\epsi$ en une borne pour $|m(x)|$}{Conversion d'une borne pour |M(x)/x|  en une borne pour |m(x)|}}

Nous allons convertir l'estimation rappelée dans la proposition  suivante, la même technique fonctionnerait pour toutes les estimations du type  $|M(x)| \ioe \epsi x$.  Rappelons que cette estimation particulière, obtenue par Cohen, Dress et El Marraki, est la meilleure de ce type connue à l'heure actuelle et que le
rang $T$  est optimal par rapport à cet $\epsi$.
\begin{prop}[Cohen, Dress, El Marraki] On  a :
\begin{equation}
\frac{|M(x)|}{x} \ioe \frac{1}{4345}  \quad (x \soe  2\,160\,535)\,. \label{M Cohen Dress El Marraki}
\end{equation}
\end{prop}
\begin{proof}
Le résultat se trouve dans le théorème 5bis de \cite[p.~62]{MobiusCohenDressMarraki}.
\end{proof}

Dans cette section,  nous obtenons à partir du résultat de Cohen, Dress et El Marraki le résultat suivant. Le rang de validité pour $m$ est le plus petit possible et pour $m_1$ nous n'avons pas les compétences et les moyens informatiques de le faire avec certitude en un temps raisonnable.
\begin{thm}  \label{m petit 4343}
On  a :
\begin{equation*}
|m(x)| \ioe  \frac{1}{4343}  \quad (x \soe 2\,160\,605)\,,
\end{equation*} 
\begin{equation*}
|m_1(x)| \ioe  \frac{1} {11\,470\,909}  \quad (x \soe 2.2 \times 10^{12}).
\end{equation*} 
\end{thm}
Auparavant le meilleur résultat de ce type était le suivant :
\begin{equation*}
|m(x)| \ioe \frac{2}{4345}  \quad (x \soe 603\,218)\,.
\end{equation*}
Il figure dans le manuscrit \cite{preprint} d'El Marraki (papier difficilement trouvable) et rapporté dans l'état des lieux de Ramaré \cite[théorème 9, p.~9]{ramareetat}. Mais avec l'identité de Balazard de l'article \cite{BalazardHal} on aurait pu atteindre la valeur du lemme suivant.  On a un facteur multiplicatif d'environ $1.1731$ par rapport à l'estimation de départ et le mieux que l'on puisse faire avec le lemme \ref{machinerie epsi} est $1+3/4-\gamma$ soit environ $1.1728$.
\begin{lem} \label{fraction petite 1}
On a la majoration suivante :
\begin{equation*}
|m(x)| \ioe  \frac{1}{3704}  \quad ( x \soe 3.5 \times 10^6  )\,.
\end{equation*}
\end{lem}

\begin{proof} 
Nous allons dans un premier temps nous occuper uniquement des valeurs  $x \soe 10^{16}$.  Par le lemme  \ref{machinerie epsi} appliqué avec $\theta=1$ et $j=0$ on a 
\begin{equation}
|m_1(x)| \ioe \sup_{T < u < x} \frac{|M(u)|}{u} (3/4-\gamma) +  \frac{1}{x^{2} } \int_{1}^{T} |M(t)|   \diff{t}+ \frac{8/3}{x}\,. \label{eq:majo-m1-R}
\end{equation}
Posons  $T=2\,160\,535$. Par la majoration triviale $|M(t)|\ioe t$ pour tout $t\soe 1$,   on a
\[
\frac{1}{x^{2} } \int_{1}^{T} |M(t)|   \diff{t}+ \frac{8/3}{x} \ioe \frac{(T/x)^2}{2}+\frac{3}{x} \underset{ \text{ car } x \soe 10^{16} }{<} 4 \times 10^{-16}\;.
 \]
Par l'estimation \eqref{M Cohen Dress El Marraki} qui est  $|M(u)|/u \ioe 1/4345$ pour $u \soe T=2\,160\,535$ injectée dans la ligne  \eqref{eq:majo-m1-R} on obtient 
\begin{equation*} 
|m_1(x)| \ioe  \frac{1}{4345}\times (3/4-\gamma)  +4 \times 10^{-16}  \ioe \frac{1}{25146} \quad (x \soe 10^{16})\,.
\end{equation*}
On en déduit l'estimation pour $m$ suivante
\begin{equation}
|m(x)| \ioe \frac{|M(x)|}{x}+|m_1(x)| \ioe \frac{1}{4345}+ \frac{1}{25146} < \frac{1}{3704} \quad (x \soe 10^{16})\,. \label{premiereetape} 
\end{equation}
 Pour les petites valeurs de $x$ les comparaisons par rapport à $\sqrt{x}$ avec le modèle \eqref{racinemHelfgott} et ceux du corollaire \ref{modeleracine}  sont aisées et terminent le raccordement à  $ x \soe 3.5 \times 10^6>(3074 \times 0.5)^2$.
\end{proof}

\begin{proof}[Démonstration du théorème \ref{m petit 4343}]
Posons $T=4.8 \times 10^6$ (ce choix est expliqué plus loin) et considérons des $x \soe 5 \times 10^{13}$.  Par le lemme  \ref{machinerie H2 CDM} appliqué avec $\theta=1$ et $j=0$ on a 
\begin{equation}
|m_1(x)| \ioe \sup_{T < u < x} |m(u)|  \times \frac{\pi^2}{6} \frac{1}{4345}   +  \frac{R_T}{x}  \; \text{ où } \;   R_T=22527.5   \int_{1}^{T} |m(t)|  \diff{t} +6\label{gaindouble}\;.
\end{equation}
 Puisque $T \ioe 7.7 \times 10^9$ l'inégalité~\eqref{racinemHelfgott} donne  $|m(t)| \ioe 0.5/\sqrt{t}$ pour tout $t$ vérifiant $3 \ioe t \ioe T $ et donc
\begin{equation}
R_T \ioe 22527.5 \times \big( 1.5+\int_{3}^{T}  \frac{0.5}{\sqrt{t}}  \diff{t} \big)+6 \ioe 4.94\times10^7 \;. \label{maj reste m} 
\end{equation}
Par le lemme \ref{fraction petite 1} on a $\sup_{T < u < x}|m(u)|  \ioe 1/3704$, en regroupant avec l'inégalité \eqref{maj reste m} dans la ligne \eqref{gaindouble} on obtient l'estimation 
\begin{equation*}
|m_1(x)| \ioe  1.231557948\times10^{-7}+\frac{4.94\times10^7}{x}  \quad (x \soe T) 
\end{equation*}
et par décroissance de la fonction majorante on arrive à 
\begin{equation}
|m_1(x)| \ioe 1/8\,119\,793  \quad (x \soe 10^{21})\,. \label{epsi I avant modele}
\end{equation}
Pour les petites valeurs de $x$ nous allons cette fois détailler les étapes. Utilisons le corollaire~\ref{modeleracine}, on a $|m_1(x)| \ioe 5.792 /\sqrt{x}$ pour $x$ entre $10^{16}$ et $10^{21}$ ce qui redescend la borne inférieure de l'estimation \eqref{epsi I avant modele} à $10^{16}$. Par le même corollaire on a aussi
\begin{equation}
|m_1(x)| \ioe \frac{0.129}{\sqrt{x}} \quad (33 \ioe x \ioe 10^{16})\,, \label{rappel racine}
\end{equation}
et $(0.129 \times 8\,119\,793)^2 <1.1 \times 10^{12}$ ramène la validité de l'inégalité de l'estimation \eqref{epsi I avant modele} à $x \soe 1.1 \times 10^{12}.$ Pour les mêmes $x$ on obtient $|m(x)| \ioe 1/4345+1/8\,119\,793<1/4342.67$ et par le corollaire~\ref{modeleracine} puis l'inégalité \eqref{racinemHelfgott} on a
\begin{equation}
|m(x)| \ioe \frac{0.701}{\sqrt{x}} \quad (3 \ioe x \ioe 10^{16})\, \quad \text{ et } \quad |m(x)| \ioe \frac{0.5}{\sqrt{x}} \quad (3 \ioe x \ioe 7.7 \times 10^9) \;, \label{rappel racine m}
\end{equation} 
le fait que $(0.701 \times 4342.67)^2<10^7$ puis $(0.5\times 4342.67)^2<4.8 \times 10^6=T$ donnent
\[
|m(x)| \ioe \frac{1}{4342.67}  \quad (x \soe 4.8 \times 10^6) \,.
\]
Puis on recommence depuis la ligne \eqref{maj reste m} avec cette nouvelle estimation de $|m|$, on a
\begin{equation}
|m_1(x)| \ioe  \dfrac{\pi^2}{6} \frac{1}{4345}  \frac{1}{4342.67}+\frac{4.94\times10^7}{x} \quad (x \soe 4.8 \times 10^6)  \label{gros gain}
\end{equation}
et donc 
\[
|m_1(x)| \ioe   \frac{1} {11\,470\,909} \quad (x \soe 10^{21})  
\]
ramenée, en suivant la même procédure que ci-dessus, à $ x \ioe 10^{16}$ puis à $x \soe (0.129 \times   11\,470\,909)^2 \soe 2.2 \times 10^{12}$ par la comparaison avec l'inégalité \eqref{rappel racine}.

On en déduit que $|m(x)| <1/4343$ pour $x>2.2 \times 10^{12}$ puis  pour $x \soe 5 \times 10^6 $ par  les comparaisons en racine \eqref{rappel racine m} et 
le fait que $(0.701 \times 4343)^2<10^7$ et $(0.5 \times 4343)^2< 4.72 \times 10^6$. Les calculs PARI/GP du lemme \ref{calculs GP} permettent de finir le raccordement entre $5 \times 10^6$ et $2\,160\,605$. 
\end{proof}

\begin{rem} \label{limite}
Les rangs sont  très importants dans nos conversions car l'implication qui semble se profiler entre l'estimation \eqref{M Cohen Dress El Marraki} de départ sur $M$ et l'inégalité \eqref{gros gain} pour $\epsi = 1/4345$ donnerait pour tout $\epsi >0$  :  
\[
\frac{|M(x)|}{x} \ioe \epsi  \quad (x\soe T_{\epsi}) \quad  \Longrightarrow \quad   |m_1(x)| \ioe \zeta(2) \epsi^2+O(\epsi^3) \quad (x\soe O(T_{\epsi}))  \,.
\]
Ce qui est en contradiction avec $\Theta=\sup_{\zeta(s)=0}\Re s <1$.

 D'ailleurs imaginons que ce supremum soit strictement inférieur à $1$. Suivant nos méthodes si l'on devait par exemple convertir $|M(x)| \ioe x^{9/10}$ pour tout $x\soe 1$ il faudrait utiliser les lemmes \ref{machinerie epsi} et \ref{machinerie H2 CDM} avec $\theta=9/10$ et ainsi obtenir :
\[
\frac{|M(x)|}{x} \ioe  \frac{1}{x^{1/10}} \quad (x\soe 1) \quad   \Longrightarrow \quad  |m_1(x)| \ioe  \frac{0.004 }{x^{1/10}} \quad (x\soe 5 \times 10^{13}) \,.
\]
Notons  que la formule sommatoire d'Abel donne ici le bon ordre de grandeur en $x$ mais avec un facteur multiplicatif  $2500$ fois plus grand que ci-dessus. En effet, on a 
\[
 m_1(x)=\int_{1}^{x} \frac{M(t)}{t^2}=\int_{1}^{\infty} \frac{M(t)}{t^2}-\int_{x}^{\infty} \frac{M(t)}{t^2}=-\int_{x}^{\infty} \frac{M(t)}{t^2} \quad \text{et} \quad \int_{x}^{\infty} \frac{t^{9/10}}{t^2} \diff{t}=\frac{10}{x^{1/10}} \,.
\]
\end{rem}

\subsection{\texorpdfstring{Conversion du théorème des nombres premiers $|M(x)/x| \ioe a/\log^j x$}{Conversion du théorème des nombres premiers |M(x)/x| < a/(log  x) puissance j}}
\subsubsection{\texorpdfstring{En partant de $|M(x)/x| \ioe a/\log x$}{En partant de |M(x)/x| < a/\log x}}

Commençons par utiliser une estimation  explicite de $M(x)=o(x)$ de Ramaré. C'est ce type d'inégalités que l'on souhaite convertir dans cette partie et
c'est la meilleure estimation de ce type obtenue à l'heure actuelle.
\begin{prop}[Ramaré] \label{ramarlog}
On a : 
\begin{equation}
\frac{|M(x)|}{x} \ioe \frac{0.013}{\log x} \quad(x \soe 97\,067)\,.
\label{eq:ramarelog}
\end{equation}
\end{prop}
\begin{proof}
Ce résultat est contenu dans le lemme 16.1 de \cite[p.~1384]{RamarExplicitMob}.
\end{proof}
La conversion sera plus dure que pour $|M(x)/x| \ioe\epsi$ car le terme constant $\epsi$ devient variable et que nos choix de rangs sont bornés (sinon le problème serait le même). Par exemple montrer un résultat isolé comme $|m_1(x)| \ioe 4 \times 10^{-8}$ pour tout $x\soe 10^{50}$ ne nous intéresse pas ici car nous ne pourrions pas redescendre la borne de validité avec les estimations sur $m_1$ dont nous disposons. 

  Remarquons  que l'estimation $|M(x)/x|\ioe 1/4345$ est plus précise que \eqref{eq:ramarelog}  jusqu'à des $x$ de l'ordre de $3 \times 10^{24}$.   Le meilleur résultat explicite de  $m(x)\ll 1/\log x $ qui était disponible avant nos travaux est dans \cite{RamarExplicitMob} et obtenu à partir de \eqref{eq:ramarelog}, on a 
\[
|m(x)| \ioe \frac{0.0144}{\log x} \quad(x\soe 96\,955)\;.
\]
Nous allons obtenir les majorations suivantes. Le rang de validité pour $m$ est le plus petit possible et pour $m_1$ nous n'avons pas les compétences et les moyens informatiques de le faire avec certitude en un temps raisonnable.
\begin{thm}\label{prop:majo-m(x)log(x)}
On a :
\begin{equation*} 
|m(x)| \ioe  \frac{0.0130073}{\log x}  \quad (x \soe 97\, 063)
\end{equation*} 
et
\begin{equation*} 
|m_1(x)| \ioe  \frac{7.265\times 10^{-6}}{\log x}  \quad (x \soe 2.15 \times 10^{11})\,. \label{propo 1 sur log }
\end{equation*}
\end{thm}
Commençons par initialiser les diminutions sur $|m_1|$, d'ailleurs on ne peut pas faire beaucoup mieux en partant de l'identité de Balazard \eqref{bal2} à travers le lemme \ref{machinerie epsi} et de l'estimation de Ramaré \eqref{eq:ramarelog}.
\begin{lem} \label{etape m log} On a la majoration suivante :
\begin{equation*}
|m_1(x)| \ioe   \frac{0.0023}{\log x} \quad ( x \soe 5\times 10^{6})\,.  
\end{equation*}
\end{lem}

\begin{proof}
Posons  $T=10^{13}$.  Par le lemme  \ref{machinerie epsi} appliqué avec $\theta=1$, $j=1$, $s=1-1/\log T$ et par la proposition \ref{ramarlog}  on a 
\begin{equation*}
|m_1(x)| \ioe \frac{ 0.1755 \times  0.013}{ \log x}  + \frac{1}{x^2} \int_{1}^{T} |M(t)| \diff{t}+ \frac{8/3}{x} \quad (x \soe T)\,.
\end{equation*}
Par l'article \cite{hurst2018computations} on a  $|M(t)| \ioe \sqrt{t}$ pour $t\ioe 10^{16}$ et donc
\[
\frac{1}{x^{2} } \int_{1}^{T} |M(t)|   \diff{t}+ \frac{8/3}{x} \ioe \frac{2T^{3/2}}{3x^2}+\frac{3}{x}  \ioe   \frac{a}{\log x}  \quad (x \soe 10^{16})
\]
avec $b=( 0.67 T^{3/2}    /10^{32} +3/10^{16}) \log (10^{16})\simeq 7.8 \times 10^{-12}$. On a $0.1725 \times  0.013+b \ioe 0.0023$ et on a prouvé 
\begin{equation*}
|m_1(x)| \ioe   \frac{0.0023}{\log x} \quad ( x \soe 10^{16})\,.  
\end{equation*}
Pour les petites valeurs de $x$ le corollaire \ref{modeleracine} donne l'estimation annoncée entre $5\times 10^{6}$ et $10^{16}$. 
\end{proof}

\begin{proof}[Démonstration  du théorème~\ref{prop:majo-m(x)log(x)}] 
Posons $T=8.2 \times 10^{25}$.  Par la corollaire \ref{modeleracine} on a $|m(t)|\ioe 0.701/\sqrt{t}$ pour $ 3 \ioe t \ioe 10^{16}$  et par le théorème \ref{m petit 4343} on a $|m(t)|\ioe 1/4343$ pour $t \soe 10^{16}$,  on obtient ainsi 
\begin{equation}
6+22727.5\int_{1}^{T}|m(t)|\diff{t} \ioe 4.254 \times 10^{26}\;. \label{RT}
\end{equation}
 Par le lemme  \ref{machinerie H2 CDM} appliqué avec $\theta=1$, $j=1$ et $\delta=1/\log T$,  par la proposition \ref{ramarlog} sur $M$ et le lemme \ref{etape m log} sur $m_1$  utilisés dans l'inégalité $|m(u)| \ioe |M(u)/u|+|m_1(u)|$ puis par l'inégalité \eqref{RT},  on a
\begin{equation}
|m_1(x)|  \ioe \frac{1}{1796.57} \frac{0.013+0.0023}{\log x}  +\frac{ 4.254 \times 10^{26}}{x}\;. \label{moulin}
\end{equation} 
Puisque la fonction $x \mapsto (\log x)/x$ est décroissante à partir de $x=\exp(1)$, on obtient que $|m_1(x)| \log x  \ioe f(x)$ pour $x\soe T$ 
avec $f$ une fonction décroissante, en calculant $f(2.5 \times 10^{42})$ on arrive à 
\begin{equation}
|m_1(x)|  \ioe   \frac{8.517 \times 10^{-6}}{\log x}   \quad (x\soe 2.5 \times 10^{42})\,. \label{first step}
\end{equation}
Par le théorème \ref{m petit 4343} on a la première inégalité des deux suivantes et la seconde provient de la résolution d'une inéquation très simple avec un calcul d'exponentielle, on a
\begin{equation*}
|m_1(x)| \ioe  \frac{1} {11\,470\,909}  \quad (x \soe 2.2 \times 10^{12}) 
 \text{ , }    \frac{1} {11\,470\,909} \ioe \frac{8.517 \times 10^{-6}}{\log x} 
\quad (3 \ioe x \ioe 2.6 \times 10^{42}) 
\end{equation*}
donc le rang de validité de l'inégalité \eqref{first step} est descendu à $2.2 \times 10^{12}$.
On recommence ensuite à partir de la ligne \eqref{moulin} avec notre meilleure estimation de $|m|$ déduite de celle de $|m_1|$ ci-dessus, on obtient
\begin{equation*}
|m_1(x)|  \ioe \frac{1}{1796.57} \frac{0.013+8.517 \times 10^{-6}}{\log x}  +\frac{ 4.254 \times 10^{26}}{x}
\end{equation*}
et par décroissance 
\begin{equation}
|m_1(x)|  \ioe  \frac{7.265 \times 10^{-6}}{\log x}  
 \quad(x \soe 1.5 \times 10^{36}) \,. \label{second step}
\end{equation}
Ensuite on compare à notre estimation constante du  théorème \ref{m petit 4343}, on a
\[
|m_1(x)| \ioe  \frac{1} {11\,470\,909}  \quad (x \soe 2.2 \times 10^{12})   \text{, }  \frac{1}{11\,470\,909} \ioe \frac{7.265 \times 10^{-6}}{\log x} \quad ( 3 \ioe x  \ioe 1.55 \times 10^{36}) 
\]
le rang de validité de \eqref{second step}  est ainsi ramené à $2.2 \times 10^{12}$ puis à $2.15 \times 10^{11}$ en comparant au modèle en $\sqrt{x}$ du lemme \ref{modeleracine}.

Passons à la fonction $m$, on a $|m(x)| \ioe |M(x)/x|+|m_1(x)|$, pour tout $x \soe 2.15 \times 10^{11}$ on obtient $0.013073/\log x$  en arrondissant au supérieur et la validité de cette estimation est ramenée à $230 000$ par les comparaisons avec $\sqrt{x}$ du  lemme \ref{modeleracine} et de l'inégalité \eqref{racinemHelfgott}  puis à $97\, 063$ par les calculs PARI/GP du lemme \ref{calculs GP}.\end{proof}

\subsubsection{\texorpdfstring{En partant de $|M(x)/x| \ioe a/\log^2 x$}{En partant de |M(x)/x| < a/\log^2 x}}

Comme dans les sections précédentes commençons par donner le meilleur résultat de ce type connu à l'heure actuelle sur $|M|$, il n'est pas optimal : par un calcul sur
ordinateur pour les petites valeurs il semble que $\sup_{u \in [1, \infty
[} \log^2 u |M(u)/u| = \log^2 31 \times 4/31  \approx 1.52 $. 
\begin{prop}[El Marraki]\label{logcarremarraki} On a :
\begin{equation*} \frac{|M(x)|}{x} \ioe \frac{362.7}{\log^2 x} \quad (x >1)\,.
\end{equation*}
\end{prop}
\begin{proof}
Ce résultat est contenu dans \cite{MarrakiMobiusIII} (corollaire en bas de la page 421).
\end{proof}
Pour la fonction $|m_1|$ nous obtenons le résultat suivant, il est  quant à lui optimal à l'arrondi près (on a $\sup_{u \in 1, \infty}  \log^2 u|m_1(u)|  =  \log^2 7 \times 29/105 =1.0458...$). Nous pourrons en fait aller un peu loin que $x>1$. 
\begin{thm}  \label{m1 log carre 119}
  On a :
\begin{equation*} | m_1(x) |  \ioe \frac{1.046}{\log^2 x} \quad (x >1) \quad \text{ et } \quad |m_1(x)| \ioe \frac{0.138 }{\log^2 x} \quad  (x \soe 671)\,.
\end{equation*}
\begin{equation*} 
|m(x)| \ioe \frac{362.84}{\log^2 x} \quad  (x > 1)\,.
\end{equation*} 
\end{thm}
Bordellès, en partant de l'estimation de la proposition~\ref{logcarremarraki} due à El Marraki et de l'identité de Balazard~\eqref{bal2} avec $\epsi_1'$, obtient dans \cite{bordelles2015some} l'estimation $|m(x)| \ioe 546/ \log^2 x$ pour tout $x>1$. En utilisant la même identité nous obtenons le résultat $|m(x)| \ioe 427/ \log^2 x$ comme étape intermédiaire (ce qui montre l'intérêt du lemme \ref{log puissance k}).
\begin{proof}[Démonstration  du théorème \ref{m1 log carre 119}]
Par l'article \cite{hurst2018computations} on a  $|M(t)| \ioe \sqrt{t}$ pour $t\ioe 10^{16}$.
Posons $T=10^{16}$,  par le lemme  \ref{machinerie epsi} appliqué avec $\theta=1$, $j=2$ et $s=1-2/\log T$ on a 
\begin{equation}
|m_1(x)| \log^2 x \ioe 362.7 \times 0.177112  + \frac{\log^2 x}{x^2} \int_{1}^T \sqrt{t} \diff{t}+ \frac{8 \log^2 x}{3x}\label{fois log carre}
\end{equation}
  et par la décroissance de la fonction  du membre de droite  de \eqref{fois log carre} on arrive à
\begin{equation*} 
|m_1(x)| \log^2 x \ioe 64.24  \quad  (x \soe 10^{16})\,.
\end{equation*}
Et puisque $m(x)=m_1(x)+M(x)/x$ pour tout $x \soe 1$, on obtient  
\begin{equation}
|m(x)|  \ioe \frac{426.94}{\log^2 x}  \quad  (x \soe 10^{16})\,. \label{premiere etape log carre}
\end{equation}
Posons désormais $T=\exp(18900)$. En utilisant le lemme  \ref{machinerie H2 CDM} appliqué avec $\theta=1$, $j=2$, $\delta=2/\log T$ et l'inégalité de Meissel $|m(t)| \ioe 1$  pour tout $t \soe 1$,  on arrive à  
\begin{equation}
|m_1(x)|  \ioe \frac{1}{2633.6} \frac{426.94}{\log^2 x}  +\frac{3000 \exp(18900)  }{x} \quad (x\soe \exp(18900))\,. \label{etape log 2}
\end{equation}
En multipliant l'inégalité par $\log^2 x$ et en utilisant la décroissance de $x \mapsto  (\log^2 x)/x$ sur l'intervalle considéré, on aboutit à
\begin{equation} 
|m_1(x) \log^2 x| \ioe 0.1622 \quad  (x \soe \exp(22000))\,. \label{first}
\end{equation} 
On abaisse le domaine de validité de \eqref{first} au-dessous de $10^{16}$ car par comparaison avec le théorème \ref{prop:majo-m(x)log(x)} et un calcul d'exponentielle
\[
|m_1(x)| \ioe \frac{7.265 \times 10^{-6}}{\log x} \quad (x \soe 3 \times 10^{11}) \text{, } \frac{7.265 \times 10^{-6}}{\log x} \ioe \frac{0.1622}{\log^2 x} \quad (3 \ioe x \ioe  \exp(22335)) \,.
\] 
On réitère la procédure avec notre meilleure estimation de $|m|$  juste avant la ligne \eqref{etape log 2}, on a
\begin{equation*}
|m_1(x)|  \ioe \frac{1}{2633.6} \frac{362.7+0.1622}{\log^2 x}  +\frac{ 30000 T}{x} \quad (x\soe T)\,.
\end{equation*}
Par décroissance on aboutit à 
\begin{equation} 
|m_1(x) \log^2 x| \ioe 0.1378  \quad  (x \soe \exp(18960) )\,.  \label{second}
\end{equation} 
On abaisse le domaine de validité de \eqref{second} au-dessous de $10^{16}$  par comparaison avec le théorème \ref{prop:majo-m(x)log(x)}, on a
\[
|m_1(x)| \ioe \frac{7.265\times 10^{-6}}{\log x} \quad (x \soe 3 \times 10^{11}) \text{ , } \frac{7.265 \times 10^{-6}}{\log x} \ioe \frac{0.1378}{\log^2 x} \quad (3 \ioe x \ioe  \exp(18967)) \,.
\] 
Ensuite on peut descendre le rang de l'inégalité à $x \soe 7000$ par la comparaison avec le lemme \ref{modeleracine} pour $x \soe 7.7 \times 10^9$ puis avec l'inégalité \eqref{modele du pauvre}. On termine  par une vérification directe entre $671$ et $7000$.

De même une vérification directe montre que 
\begin{equation*} 0 \ioe   m_1(x)  \log^2 x  \ioe 1.046 \quad (1 \ioe x \ioe 671 )\,,
\end{equation*}
plus précisément le maximum est atteint pour $x=7$ et vaut $(29/105) \log^2 7$.

On obtient le résultat sur $|m|$ pour $x \soe 671$ à l'aide de $|m(x)| \ioe |M(x)/x| +|m_1(x)|$, et l'on termine par une  vérification directe.
\end{proof}

 \begin{proof}[Démonstration du théorème \ref{thm A}]
 C'est une conjonction des résultats concernant la fonction $m$ du théorème \ref{m petit 4343} p. \pageref{m petit 4343},   du théorème    \ref{prop:majo-m(x)log(x)}   p. \pageref{prop:majo-m(x)log(x)} et du théorème \ref{m1 log carre 119} p. \pageref{m1 log carre 119}.
 \end{proof}

\subsubsection{Démonstrations des deux familles d'identités du théorème \ref{thm1}}
Nous allons établir la démonstration du théorème \ref{thm1}. Tout repose sur des interversions de sommes et d'intégrales puis sur l'inversion de M\"obius.
\begin{lem}\label{echange}
Soit $\delta : [1,\infty[ \rightarrow \CC$ une fonction localement intégrable, pour tout $x\soe 1$ nous avons : 
\begin{equation}
 \int_{1}^{x}M(x/t) \delta(t) \frac{\diff{t}}{t}=\int_{1}^{x} \sum_{n \ioe u} \mu(n) \delta(u/n) \frac{\diff{u}}{u}\;.\label{Mmusurn}
\end{equation}

En particulier, pour tout $x\soe 1$, nous avons :
\begin{equation}
\int_{1}^{x}M(x/t) \frac{1}{t} \sum_{k \ioe t} \delta\Big(\frac{t}{k}\Big)\diff{t}=\int_{1}^{x} \delta(u) \frac{\diff{u}}{u}\;. \label{Mobintegral} 
\end{equation}
\end{lem}

\begin{proof} 
En échangeant les sommes et les  intégrales dans ce qui suit, puis en effectuant
le changement de variable $u=t n$, on obtient : 
\begin{align*} 
\int_{1}^{x}M(x/t) \delta(t)  \frac{\diff{t}}{t} &=\sum_{n \ioe x} \mu(n) \int_{1}^{x/n} \delta(t)   \frac{\diff{t}}{t} \\
&=\sum_{n \ioe x} \mu(n) \int_{n}^{x}  \delta(u /n) \frac{\diff{u}}{u} 
= \int_{1}^{x} \sum_{n \ioe u} \mu(n) \delta(u/n) \frac{\diff{u}}{u}\;,
\end{align*}
ce qui est la première égalité annoncée en~\eqref{Mmusurn}.

Pour la seconde formule, considérons la fonction $\delta_1(t)= \sum_{k \ioe t} \delta(t/k)$. Elle est localement intégrable sur $[1,x]$, on peut  utiliser 
la formule~\eqref{Mmusurn}, qui a été précédemment démontrée, pour la fonction $\delta_1$. On a 
\begin{equation*}
\int_{1}^{x}M(x/t) \delta_1(t) \frac{\diff{t}}{t}=\int_{1}^{x} \sum_{n \ioe u}\mu(n) \delta_1(u/n) \frac{\diff{u}}{u}\;.
\end{equation*} 
Pour finir, on remarque que la somme intégrée peut se calculer par inversion de M\"obius, 
\begin{equation*}
\sum_{n \ioe u} \mu(n) \delta_1(u/n)= \sum_{n \ioe u}\mu(n) \sum_{k \ioe u/n}\delta\big(u/(nk)\big)= \delta(u)\;,
\end{equation*}
ce qui achève la démonstration.
\end{proof}
On démontre de la même manière le lemme suivant.
\begin{lem}
Soit $\delta : [1,\infty[ \rightarrow \CC$ une fonction localement intégrable, pour tout $x\soe 1$ nous avons :
\begin{equation}
\int_{1}^{x}m(x/t) \delta(t) \frac{\diff{t}}{t^2}=\int_{1}^{x} \sum_{n \ioe u} \mu(n) \delta\left(\frac{u}{n}\right) \frac{\diff{u}}{u^2}\;. \label{echangem}
\end{equation}

En particulier, pour tout $x\soe 1$, nous avons :
\begin{equation}
\int_{1}^{x}m(x/t) \sum_{k \ioe t} \delta(t/k)\frac{\diff{t}}{t^2}=\int_{1}^{x} \delta(u) \frac{\diff{u}}{u^2}\;. \label{echangemsomme}
\end{equation}
\end{lem}

\begin{proof}[Démonstration du théorème \ref{thm1}]
En intervertissant la somme et l'intégrale on obtient 
\begin{equation}
\int_{1}^{x}M(x/t)  \diff{t} =\sum_{n \ioe x} \mu(n) \int_{1}^{x/n}   \diff{t}=\sum_{n \ioe x} \mu(n)(x/n-1)=xm(x)-M(x)=xm_1(x)\,. \label{M integral}
\end{equation}
Puisque $G$ s'exprime comme la différence $ G(t)=1 -\frac 1t\sum_{n\leqslant t} g\big(\frac{n}{t}\big) $ on trouve
\begin{align*}
  \int_1^x M(x/t) G(t)  \diff{t}= \int_1^x M(x/t)  \diff{t}&-\int_1^x M(x/t) \frac 1t\sum_{n\leqslant t} g\Big(\frac{n}{t}\Big)  \diff{t} \\
 =x m_1(x)&- \int_{1}^{x}  g(1/u) \frac{\diff{u}}{u}
\end{align*}
où la première intégrale est obtenue par  \eqref{M integral} et la deuxième par l'inversion de M\"obuis du  lemme \ref{echange} avec l'égalité  \eqref{Mobintegral} appliquée pour $\delta(u)=g(1/u)$. Le changement de variable $y=x/u$ termine la démonstration de la première identité du théorème \ref{thm1} à une division par $x$ près.

Passons aux identités avec $m$, on a l'égalité suivante
 \begin{equation}
\int_{1}^{x}m(x/t)  \frac{\diff{t}}{t^2}=\sum_{n \ioe x} \frac{\mu(n)}{n} \int_{1}^{x/n} \frac{\diff{t}}{t^2}=\sum_{n \ioe x} \frac{\mu(n)}{n} \left( 1-  \frac{n}{x} \right)=m_1(x)\,. \label{m integral}
\end{equation}
Puisque $H$ s'exprime comme la différence $ H(t)=1 -\sum_{n\leqslant t} h\big(\frac{n}{t}\big) $ on trouve
\begin{align}
  \int_1^x m(x/t) H(t)  \frac{\diff{t}}{t^2}= \int_1^x m(x/t) \frac{\diff{t}}{t^2}&-\int_1^x m(x/t) \sum_{n\leqslant t} h\Big(\frac{n}{t}\Big)  \frac{\diff{t}}{t^2} \nonumber \\
 =x m_1(x)&- \int_{1}^{x}  h(1/u) \frac{\diff{u}}{u^2}  \label{avant soustraction}
\end{align}
où la première intégrale est obtenue par  \eqref{m integral} et la deuxième par l'inversion de M\"obuis   \eqref{echangemsomme} appliquée pour $\delta(u)=h(1/u)$. Le changement de variable $y=x/u$ donne
\begin{equation}
\int_{1}^{x}  h(1/u) \frac{\diff{u}}{u^2}=\int_{1/x}^{1}  h(y) \diff{y}=\int_{0}^{1}  h(y) \diff{y}-\int_0^{1/x}  h(y) \diff{y} 
\end{equation}
on conclut puisque  $\int_{0}^{1}  h(y) \diff{y}=0$ par hypothèse.  \end{proof}

\section{\texorpdfstring{De $m_1$ vers $\check{m}$}{De m1 vers check(m)}}

Nous suivrons une autre méthode que les articles \cite{BalazardHal} et \cite{RamarExplicitMob} pour étudier $\check{m}$, les deux auteurs passaient par des identités comportant la fonction $M$ et nous utiliserons la fonction $m_1$.

On peut réfléchir par analogie entre le point 1) du théorème \ref{thm1}  et la proposition \ref{mchliss}: les identités et les conversions  pour le couple $(m,M)$ qui passent par des fonctions $G$ vont pouvoir se réaliser pour le couple $(\check{m},m_1)$. Nous utiliserons ces identités et les bonnes estimations obtenues pour $|m_1|$ pour aboutir à des estimations de $|\check{m}-1|$ proches. Ce qui montre l'intérêt d'avoir étudié plus précisément des majorations de  $|m_1|$ et pas uniquement de s'être focalisé sur $|m|$.
 \begin{prop}\label{mchliss}Soit $g :[0,1] \rightarrow \CC$ une fonction intégrable avec $\int_{0}^{1}g(y) \diff{y}=1$ et soit $x\soe 1$, on a :
\begin{equation*}
(\check{m}(x)-1)-m_1(x) =\int_{1}^{x}m_{1}(x/t) G(t) \frac{\diff{t}}{t} -\frac{1}{x}\int_{1/x}^{1}\frac{g(y)}{y} \diff{y}-\int_{0}^{1/x}g(y) \diff{y} 
\end{equation*}
\begin{equation*}
\text{où } \quad  G(t)=1 -\frac 1t\sum_{n\leqslant t} g\Big(\frac{n}{t}\Big) \,.
\end{equation*}
\end{prop}

\begin{proof}
Commençons par prouver l'identité suivante
\begin{equation}
\check{m}(x)-1 =\int_{1}^{x}m(x/t) G(t) \frac{\diff{t}}{t}-\int_{0}^{1/x}g(y) \diff{y} \quad (x \soe 1)\; . \label{checkm par m}
\end{equation}
Par l'égalité \eqref{echangemsomme}  basée sur l'inversion de M\"obius appliquée avec $\delta(u)=g(1/u)$ on a
\begin{align*}
\int_{1}^{x}m(x/t) G(t) \frac{\diff{t}}{t}&=\int_{1}^{x}m(x/t) \frac{\diff{t}}{t}-\int_{1}^{x}m(x/t) \sum_{n \ioe t} g(n/t)\frac{\diff{t}}{t^2}\\
&=\sum_{n \ioe x} \frac{\mu(n)}{n} \log(x/n)-\int_{1}^{x}g(1/u) \frac{\diff{u}}{u^2} 
\end{align*}
mais puisque $\int_{0}^{1}g(y) \diff{y}=1$, un changement de variable donne 
\[
\int_{1}^{x}g(1/u)u^{-2}\diff{u}=\int_{1/x}^{1}g(y) \diff{y}=1-\int_{0}^{1/x}g(y) \diff{y},
\]
 ce qui termine de prouver l'égalité \eqref{checkm par m}. On peut alors  calculer 
\[
\int_{1}^{x}m_{1}(x/t) G(t) \frac{\diff{t}}{t}=\int_{1}^{x}m(x/t) G(t) \frac{\diff{t}}{t}-\int_{1}^{x}\frac{M(x/t)}{x/t} G(t) \frac{\diff{t}}{t}
\]
en soustrayant l'égalité \eqref{checkm par m} de l'identité de la partie 1) du théorème \ref{thm1}, on obtient
\begin{equation*}
\int_{1}^{x}m_{1}(x/t) G(t) \frac{\diff{t}}{t}=\check{m}(x)-1 +\int_{0}^{1/x}g(y) \diff{y}-\Big(m_1(x)-\frac{1}{x} \int_{1/x}^{1}\frac{g(y)}{y} \diff{y} \Big)
\end{equation*}
ce qui termine la démonstration.
\end{proof}
La proposition ci-dessus est apparentée au théorème \ref{thm1} et il était aussi possible de spécialiser une proposition générale de notre thèse \cite[p.~17]{daval2019identites} pour obtenir toutes ces identités. 

En prenant la fonction $g_1(y)=4(1-y^2)y$ issues des travaux de Balazard on obtient le résultat suivant.
\begin{lem} \label{machinerie G1 check}
Il existe une fonction $G_1 : \RR_+ \rightarrow \RR_+$ telle que  pour tous $T,x,j,\theta  \in\RR$ tels que $1< T \ioe x$, $j \soe 0$ et $\theta>-1$, on a
\begin{equation*} 
|\check{m}(x)-1| \ioe \sup_{T < u < x} \big( u^{-\theta+1} |m_1(u)| \log^j u \big) \times \Big(1+\int_{1}^{\infty} G_1(t) t^{-\theta +\frac{j}{\log T}}\diff{t}\Big) \times \frac{x^{\theta-1}}{\log^j x}  +  \frac{R_T(x)}{x} \label{boite G1 check}
\end{equation*}
avec $R_T(x)= 8/3+  (1/x) \int_{1}^{T} |M(t)|   \diff{t}  $. \newline

 Pour tout $s \in\CC$ tel que $\Re s>-1$, $s\neq 1$, on a 
\begin{equation*}
1+\int_{1}^{\infty} G_1(t) t^{-s}\diff{t}
= \frac{s}{s-1}- \frac{8\zeta(s)}{(s+1)(s+3)}\;.  \quad \text{Et aussi } 1+ \int_{1}^{\infty} G_1(t) t^{-1}\diff{t}=\frac{7}{4}-\gamma\;.  \label{Muntz G1 check}
\end{equation*}
\end{lem}
\begin{proof}
 Posons $g_1(y)=4(1-y^2)y$. En comparant la première formule du théorème \ref{thm1} et l'identité de la proposition \ref{mchliss} on comprend que la démonstration est identique à celle du lemme \ref{machinerie epsi} en remarquant de plus qu'un calcul de primitive donne
\[
 0 \ioe \frac{1}{x}\int_{1/x}^1g_1(y)/y \diff{y}+\int_{0}^{1/x}g_1(y) \diff{y}=\frac{8x^3-6x^2+1}{3x^4}\ioe \frac{8/3}{x}\;.
\] 
Le  $+1$ devant les intégrales provient de  $|\check{m}(x)-1| \ioe |\check{m}(x)-1-m_1(x)|+|m_1(x)|$. \end{proof}

\begin{thm} \label{thm check}
 Sur des intervalles finis,  on a les estimations suivantes : 
\begin{equation}
|\check{m}(x)-1| \ioe \frac{0.16}{\sqrt{x}}   \quad ( 10^{9} \ioe x \ioe 10^{16}) \quad\text{et}\quad |\check{m}(x)-1| \ioe \frac{7.1}{\sqrt{x}}   \quad (10^{16} \ioe x \ioe 10^{21})\;. 
\label{check m racine}
\end{equation}
Au voisinage de l'infini, on a :
\begin{align}
|\check{m}(x)-1| &\ioe    \frac{1}{9\,780\,919} \quad ( x \soe  2.5 \times 10^{12})\, , \label{check m epsilon}\\
|\check{m}(x)-1| &\ioe    \frac{8.55 \times 10^{-6}}{\log x} \quad ( x \soe 2.5 \times 10^{11})\, , \label{check m un sur log}\\
|\check{m}(x)-1| &\ioe    \frac{0.162}{\log^2 x} \quad ( x \soe 3) \,. \label{check m un sur log carre}
\end{align}
\end{thm}
\begin{proof}
Posons $T=5\times 10^{6}$, par l'inégalité \eqref{racineMHurst} on trouve $\int_{1}^{T} |M(t)| \diff{t} \ioe 4.26 \times 10^9 $. Appliquons le lemme  \ref{machinerie G1 check} avec $\theta=0.5$, $j=0$ et $s=0.5$ on obtient
\[
 |\check{m}(x)-1| \ioe   \sup_{T < u < x} |m_1(u)| \sqrt{u} \times \frac{1.2254}{\sqrt{x}}  +\frac{8/3}{x}+  \frac{4.26 \times 10^9}{x^2}=\frac{f(x)}{\sqrt{x}}
 \]
 par le corollaire \ref{modeleracine}  on aboutit à $f(10^9) \ioe 0.16$ et par décroissance de $f$ on prouve ainsi la première inégalité du théorème.
 
Pour tout $x\ioe 10^{16}$  par le corollaire \ref{modeleracine} et par la proposition \ref{racine autres} on a $x|m(x)| \ioe \sqrt{2x} $ et $|M(x)| \ioe \sqrt{x} $, donc  $x|m_1(x)| \ioe 3\sqrt{x} $ pour $x\ioe 10^{16}$ et ainsi on a $|m_1(x)| \sqrt{x} \ioe 5.792 $ pour  $1 \ioe x \ioe 10^{21}$. On conclut facilement de la même manière que précédemment en prenant $T=1$.

Posons $T=2.2 \times 10^{12}$,  par l'inégalité \eqref{racineMHurst} on trouve $\int_{1}^{T} |M(t)| \diff{t} \ioe 1.25 \times 10^{18} $. Par le lemme  \ref{machinerie G1 check} appliqué avec $\theta=1$ et $j=0$ on a 
\begin{equation}
|\check{m}(x)-1| \ioe \sup_{T < u < x} |m_1(u)| (7/4-\gamma) + \frac{8/3}{x}+ \frac{1.25 \times 10^{18}}{x^2}\;.
\end{equation}
On conclut par le théorème \ref{m petit 4343} pour le supremum et par décroissance de la majoration à l'estimation $|\check{m}(x)-1| \ioe 9780919^{-1}$ pour $x\soe 10^{16}$. Par le résultat \eqref{check m racine} on redescend la borne à $x=(0.16 \times 9780919)^2 \ioe 2.5 \times 10^{12}$.

Posons $T=2.15 \times 10^{11}$,  par l'inégalité \eqref{racineMHurst} on trouve $\int_{1}^{T} |M(t)| \diff{t} \ioe 3.8 \times 10^{16} $. Par le lemme  \ref{machinerie G1 check} appliqué avec  $\theta=1$, $j=1$  et $s=1-1/\log T$ on a 
\begin{equation}
|\check{m}(x)-1| \ioe \sup_{T < u < x} \log u |m_1(u)| \frac{1.17582}{\log x} + \frac{8/3}{x}+ \frac{3.8 \times 10^{16}}{x^2} \;.
\end{equation}
On conclut en utilisant le théorème \ref{prop:majo-m(x)log(x)} pour le supremum et par décroissance de la majoration obtenue à $\log(x)|\check{m}(x)-1| \ioe 8.55 \times 10^{-6}$ pour $x\soe 10^{16}$. Par le résultat \eqref{check m racine} on redescend la borne à $x= 2.5 \times 10^{11}$.

Par le lemme  \ref{machinerie G1 check} appliqué avec $T=10^{100}$ $\theta=1$, $j=2$, $s=1-2/\log T$ et par  la majoration triviale  $\int_{1}^{T} |M(t)| \diff{t} \ioe 0.5 T^2$, on a
\begin{equation}
|\check{m}(x)-1| \ioe \sup_{T < u < x} \log^2 u |m_1(u)| \frac{1.1735}{\log^2 x} + \frac{8/3}{x}+ \frac{0.5 \times 10^{200}}{x^2}
\end{equation}
et on conclut en utilisant le théorème \ref{m1 log carre 119} pour le supremum et par par décroissance à l'inégalité $(\log^2 x)|\check{m}(x)-1| \ioe 0.162$ pour $x\soe \exp(100)$. Par le résultat \eqref{check m un sur log} on redescend la borne à $x=2.5 \times 10^{11}$. Puis par le résultat \eqref{check m racine} et une conversion simple de l'inégalité \eqref{modele du pauvre} on redescend la borne à $x= 5000$, résultat rapporté à $x=3$ par vérification directe. \end{proof}
Le meilleur résultat  sur $|\check{m}(x)-1|$ était dans l'article \cite{ramare2013explicit} (voir  corrigendum  \cite{Corrig}), on a
\[
|\check{m}(x)-1| \ioe    \frac{0.0014}{\log x} \quad ( x \soe 9950) \,.
\]
Dans cet article Ramaré étudie également $\check{\check{m}}(x)= \sum_{n \ioe x} (\mu(n)/n) \log^2(x/n)$, cette fonction est utilisée par Helfgott \cite{helfgott2012minor} et Zuniga Alterman \cite{alterman2020logarithmic}. Nous avons prouvé avec l'identité \eqref{checkm par m} et dans notre thèse \cite[proposition 50]{daval2019identites}  des versions  explicites de
\begin{align}
\check{m}(x)-1=\int_{1}^{x}m(x/t)G(t)&\frac{\diff{t}}{t}+O(1/x)\\
\check{\check{m}}(x)-2\log x+2\gamma =2\int_{1}^{x}(\check{m}(x)-1)G(t)&\frac{\diff{t}}{t}+O(1/x) \,,
\end{align} 
où $G$ a la même définition que dans le théorème \ref{thm1} et dans la proposition \ref{mchliss}. Il serait donc intéressant d'obtenir des fonctions $G$ avec $\int_{1}^{\infty} |G(t)|t^{-1} \diff{t}$ très petit. En prenant la fonction de Balazard on a environ $0.17$ pour cette intégrale et on aimerait obtenir des résultats de l'ordre de la fonction $H$ de Cohen-Dress-El Marraki où l'on a environ $0.0004$. 

Nous avons montré dans notre thèse que cet objectif est théoriquement possible et en pratique on peut même utiliser une fonction $H$ pour faire une fonction $G$ : avec la fonction $H$ de l'article de Costa Peirera  \cite{CostaPPsiM} nous sommes en mesure d'assurer $0.003$ avec $K=8020$. Nous essayons encore d'améliorer ces valeurs, dans les deux cas (fonctions $H$ et $G$) l'étude de $S_K(t)=\sum_{k \ioe K}\mu(k)\{t/k\}$ est déterminante.

C'est pour ces raisons que nous ne faisons pas ici de conversions pour $\check{\check{m}}$, on aurait  les même encadrements que pour $\check{m}$ avec un facteur multiplicatif de $0.35$ alors que nous obtiendrons au moins un facteur multiplicatif de $0.01$ avec un peu de travail. Précisons également que  dans l'article \cite{ramare2013explicit} les estimations pour $\check{\check{m}}$ sont moins bonnes que pour $\check{m}$.  De plus dans un futur papier nous allons étudier les sommes  $\sum_{n \ioe x} (\mu(n)/n) \log^k(x/n)$ pour tout nombre entier $k \soe 0$.

\begin{proof}[Démonstration du théorème \ref{thm B}]
 Il s'agit des 3 dernières majorations du théorème \ref{thm check}.
\end{proof}

 Comme annoncé dans le résumé on obtient à partir de la troisième inégalité du théorème \ref{thm B} que $\sup_{x \soe 1} (\log^2 x) |\check{m}(x)-1|$ est atteint pour $x=\exp((4-2\log 2)/3)\simeq 2.39$ et vaut exactement $2(2-\log 2)^3/27 \simeq 0.1653$.
\section{Analogue d'un résultat de Odlyzko et te Riele}
Nous allons maintenant tirer parti de conversions de $m$ vers $M$. On étudiant un peu sur son ordinateur la fonction $|m(x)| \sqrt{x}$ il semble que le supremum soit atteint en $x=2^-$ et vaille $\sqrt{2}$ et qu'ensuite la fonction soit même bien plus petite. Ce qui est confirmé par nos modèles du corollaire \ref{modeleracine} en effet nous avons prouvé que  $|m(x)| \sqrt{x} \ioe 0.701$ pour tout $3 \ioe x \ioe 10^{16}$. En fait en prouvant le théorème \ref{thm D} présenté dans l'introduction nous allons voir  que ces inégalités ne peuvent subsister à l'infini.

Nous allons utiliser une identité du même type que  l'identité \eqref{bal2} de Balazard mais qui exprime $M$ en fonction de $m$ au lieu de l'inverse. Le lemme suivant est à comparer au lemme \ref{machinerie epsi} (notons que le nombre $2.1$ dans l'expression de $R_T$ ci-dessous semble pouvoir être remplacé par $1$ mais c'est sans importance dans cet article).
\begin{lem} \label{machinerie H1}
Il existe une fonction $H_1 : \RR_+ \rightarrow \RR_+$ telle que pour tous $T,x,j,\theta \in\RR$ tels que $1 \ioe T \ioe x$ et $\theta>-1$, on a :
\begin{equation*}
|m_1(x)| \ioe \sup_{T < u < x} u^{-\theta} u|m(u)| \times \int_{1}^{\infty}H_1(t)  t^{-1-\theta}\diff{t} \times x^{\theta-1}\\
 +  \frac{R_T(x)}{x} \;.
\end{equation*}
avec $R_T(x)=2+(2.1/x) \int_{1}^{T} u|m(u)|   \diff{u}$. \newline

 Pour tout $s \in\CC$ tel que $\Re s>-1$, $s\neq 0$, $s\neq 1$ on a  :
\begin{equation*}
\int_{1}^{\infty}H_1(t) t^{-s}\frac{\diff{t}}{t}= \frac{1}{s} -\frac{4(s-1)(5s+9) \zeta(s)}{3s(s+1)(s+2)(s+3)}\;. \quad \text{Et on a }  \int_{1}^{\infty}H_1(t) \frac{\diff{t}}{t}=2\zeta'(0)+\frac{41}{18} \; . \end{equation*}
\end{lem}

\begin{proof}
Nous allons étudier la somme
\begin{equation}
H_1(t)=1 -\sum_{n\leqslant t} h\Big(\frac{n}{t}\Big) \quad  \text{ avec } \quad  h_1(y)=\frac{2}{3}(1-y^2)(8y-3)  \label{def H1}
\end{equation}
et prouver que
\begin{equation}
0 \ioe H_1(t) \ioe \frac{2.1}{t} \,.  \label{prop H1}
\end{equation}  
La formule d'Euler-Maclaurin donne (voir l'appendice B p. 495 de \cite{montgomery2006multiplicative})
\begin{equation}
H_1(t)=\frac{\frac{20}{6}(\{t \}^2-\{t\})+1}{t}+ \frac{r(t)}{6t^2} \text{ , } |r(t)| \ioe  \max_{u\in[0,1]}|B_3(u)| |h_1'''|=\frac{\sqrt{3}}{36}\times 32 \ioe 1.56 \label{dvp H}
\end{equation}
où $B_3$ est le polynôme de Bernoulli unitaire de degré $3$. Le minimum de $t \mapsto \frac{20}{6}(\{t \}^2-\{t\})+1$ vaut $ \frac{20}{6} \times (-\frac{1}{4})+1=\frac{1}{6}$, ce qui assure la positivité de $t H_1(t)$ pour $t \soe 1.56$.
Par définition  de $H_1$ en \eqref{def H1} on a également pour $t\in [1,\,2[$ l'égalité
\begin{equation*}
3t^3H_1(t)=3t^3(1-h_1(1/t))=9{t}^{3}-16{t}^{2}-6t+16 
\end{equation*} 
et ce polynôme de degré 3 n'a qu'une seule racine réelle qui de plus est négative. On a ainsi montré la positivité dans  \eqref{def H1}.
L'équation \eqref{dvp H} montre également que $|H_1(t)|t \ioe \frac{20}{6} \times \frac{1}{4} +1+ \frac{1.56}{6} \ioe 2.1 $. On a  ainsi montré la majoration dans \eqref{def H1}.

Maintenant on suit la même démarche que dans la démonstration du lemme \ref{machinerie epsi}. Puisque $\int_{0}^{1} h_1(y) \diff{y}=0$, que $h'_1$ continue et $h_1(1)=h'_1(1)=0$ on peut utiliser la formule de M\"untz du paragraphe II.11 du livre de Titchmarsh \cite{titchmarsh1986theory}) appliquée avec $F(y)=h_1(y) \mathrm{1}_{[0,1]}(y)$ on obtient
\begin{equation*}
\int_{1}^{\infty}(1-H_1(t)) t^{-s}\frac{\diff{t}}{t}=\int_{0}^{\infty}(1-H_1(t)) t^{-s}\frac{\diff{t}}{t}= \zeta(s) \int_{0}^{1} h_1(y) y^{s-1} \diff{y}  \quad   (0< \Re(s) <1) 
\end{equation*}
et pour $\Re s>0$  on a $\int_{1}^{\infty}t^{-s-1} \diff{t}=1/s $, donc 
\begin{equation}
\int_{1}^{\infty}H_1(t)t^{-s}\frac{\diff{t}}{t}= \frac{1}{s} \bigg[1-\zeta(s)\frac{4(s-1)(5s+9)}{3(s+1)(s+2)(s+3)}\bigg] \quad  (0< \Re(s) <1) \;. 
\label{Mellin H deux}
\end{equation}
D'après les inégalités \eqref{prop H1}, le membre de gauche de \eqref{Mellin H deux} est holomorphe pour $\Re s>-1$ donc par prolongement analytique
l'identité \eqref{Mellin H deux} reste vraie pour les mêmes valeurs de $s$ en supposant $s\neq 0$ et $s\neq 1$. En passant à la limite pour $s \to 0$ dans \eqref{Mellin H deux}, on reconnaît un
nombre dérivé. En effet posons $f(s)=\zeta(s)\frac{4(s-1)(5s+9)}{3(s+1)(s+2)(s+3)}$ on a $f(0)=1$ et 
\[ 
\lim_{s \to 0} \frac{f(0+s)-f(0)}{s}=f'(0)=\zeta'(0) \times (-2) +\zeta(0)\times\frac{41}{9}=-\left( 2\zeta'(0)+\frac{41}{18}\right)\;.
\]
Ainsi la limite pour $s \to 0$ de \eqref{Mellin H deux} donne :
\[
\int_{1}^{\infty}H_1(t) \frac{\diff{t}}{t}=2\zeta'(0)+\frac{41}{18}\;.
\]
Le théorème \ref{thm1} appliqué avec $h_1(y)=\frac{2}{3}(1-y^2)(8y-3)$ donne
\begin{equation}
m_1(x)=m(x)-\frac{M(x)}{x} =\int_{1}^{x}m(x/t) H_1(t) \frac{\diff{t}}{t^2} -\int_0^{1/x} h_1(y) \diff{y} \;. \label{identite H1} 
\end{equation}
Puisque tout est explicable, il est facile de montrer que $x \mapsto -x\int_0^{1/x} h_1(y) \diff{y}$ est une fraction rationnelle qui est positive, croissante et tend vers $2$ quand $x \to \infty$. Ainsi $|\int_0^{1/x} h_1(y) \diff{y}| \ioe 2/x$ pour tout $x\soe 1$.

Le reste de la démonstration de la majoration avec le découpage de l'intégrale dans \eqref{identite H1}  sur les intervalles $[1,\,x/T]$ et $[x/T,\,x]$   est identique à  la démonstration du lemme \ref{machinerie epsi} et ne présente aucune nouvelle difficulté. \end{proof}
Pour convertir des résultats sur des limites supérieurs nous avons besoin que le supremum porte sur un intervalle qui part à l'infini avec $x$, la majoration \eqref{conversion racine} n'aurait pas suffit.
\begin{prop}\label{prop:conversion-m-vers-M-racince}
On a  : 
\begin{equation}
|m_1(x)|\sqrt{x}\ioe b \sup_{t \ioe x}|m(t)|\sqrt{t}+\frac{2}{\sqrt{x}} \quad \text{ 
avec } \quad  b=2+\frac{368}{315}\zeta(1/2) = 0.2939...  \label{conversion racine2}
\end{equation}
En particulier, 
\begin{equation}
\frac{|M(x)|}{\sqrt{x}} \ioe (1+b)\sup_{t \ioe x}|m(t)|\sqrt{t}+\frac{2}{\sqrt{x}}\;.  \label{conversion racine}
\end{equation}
Pour tout $\epsi$ vérifiant $0 <\epsi < 3/4$, on a :
\begin{equation}
\frac{|M(x)|}{\sqrt{x}} \ioe (1+b)\sup_{x^{\frac{3}{4}-\epsi} \ioe t \ioe x}|m(t)|\sqrt{t}+\frac{2}{\sqrt{x}}+\frac{1.05}{x^{2\epsi}}\;.  \label{conversion racine localise}
\end{equation}
\end{prop}
\begin{rem}
L'insertion d'une majoration du type $|m(t)| \ioe A/ \sqrt{t}$ avec $A>0$
dans l'identité $
xm_1(x)=\int_{1}^xm(t)\diff{t}
$ fournit   
\[
|m_1(x)| \ioe \frac{2A}{\sqrt{x}}\;,
\] 
tandis que l'estimation \eqref{conversion racine2} fait gagner plus d'un facteur $6$ avec 
\[
|m_1(x)| \ioe \frac{0.3A}{\sqrt{x}}+\frac{2}{x}\;.
\]
\end{rem}

\begin{proof}
Le lemme \ref{machinerie H1} appliqué avec $T=1$, $\theta=1/2$ et $s=1/2$  donne directement le résultat \eqref{conversion racine2}.

Soit $a\in[0,1]$ un paramètre que l'on fixera plus bas. Le même lemme \ref{machinerie H1} appliqué avec cette fois $T=x^a$ et l'inégalité de Meissel $|m(u)| \ioe 1$ pour tout $u \soe 1$ donne, après multiplication par $\sqrt{x}$, l'inégalité
\begin{equation}
\sqrt{x } \frac{R_T(x)}{x}=\frac{2}{\sqrt{x}}+ \frac{2.1}{x^{1.5}} \int_{1}^{T} u|m(u)|   \diff{u} \ioe \frac{2}{\sqrt{x}}+ \frac{2.1 \times  T^2}{x^{1.5}}=\frac{2}{\sqrt{x}}+\frac{1.05}{x^{1.5-2a}}
\end{equation}
Ce qui permet de conclure en prenant $a=3/4-\epsi$.
\end{proof}
Dans \cite{odlyzko1984disproof} Odlyzko et te Riele ont prouvé que $\varlimsup |M(x)|/\sqrt{x}>1$ réfutant ainsi l'hypothèse de Mertens. Leur travail 
a été poursuivi par Hurst qui a donné une  valeur significativement plus grande que $1$ qui permet, combinée à la proposition
\ref{prop:conversion-m-vers-M-racince}, d'établir que l'on a également 
$\varlimsup  |m(x)|\sqrt{x}>1$ et même plus. 
\begin{proof}[Démonstration du théorème \ref{thm D}]
D'après  \cite{hurst2018computations} on a $\varlimsup  |M(x)|/\sqrt{x}>1.837625 $, on conclut avec l'inégalité \eqref{conversion racine localise} puisque  $1.837625/(1+b)>1.42018$.
\end{proof}

En utilisant l'identité \eqref{bal2} de Balazard  nous pouvons montrer que (voir \eqref{maj bal} plus bas) 
\begin{equation*} 
x|m_1(x)| \ioe \frac{1}{3}  \sup_{t \ioe x}|M(t)| + \frac{8}{3} \quad (x \soe 1)
\end{equation*} 
avec rappelons-le $xm_1(x) = xm(x)-M(x)$.
Cela permet d'obtenir des estimations de $|m|$ à partir d'estimations de $|M|$. Nous allons montrer que réciproquement, on peut obtenir des estimations de $|M|$ à partir d'estimations de $|m|$. 
\begin{cor} \label{autre sens}
On a l'inégalité :
\begin{equation*} 
x| m_1(x)| \ioe 0.44  \sup_{t \ioe x}t|m(t)| + 2 \quad (x \soe 1)\,.
\end{equation*}   
En particulier, on a 
\[
|M(x)| \ioe
1.44\sup_{t \ioe x}t|m(t)| + 2 \quad (x \soe 1)\,.
\] 
\end{cor}
\begin{proof}
Le lemme \ref{machinerie G1 check} appliqué avec $T=1$, $\theta=0$ et $2\zeta'(0)+\frac{41}{18}=0.4399...$  donne  les deux résultats puisque $|M(x)| \ioe x|m(x)| + x|m_1(x)|$. \end{proof}
\begin{rem}
La formule sommatoire d'Abel 
\[
M(x) =\sum_{n\ioe x}\frac{\mu(n)}{n} n =  m(x)x- \int_{1}^x m(t) \diff{t}
\]
suivie de l'inégalité triangulaire n'aboutirait qu'à la majoration 
\[
x|m_1(x)| \ioe \log x \sup_{t \ioe x}t|m(t)|\;. 
\] 
\end{rem}

\begin{proof}[Démonstration  du théorème \ref{thm C}]
En prenant $j=0$, $\theta=0$,  $T \to 1$ et $s=0$ dans le lemme \ref{machinerie epsi} (ou plus directement l'identité de Balazard \eqref{bal2} et le fait que $\int_{1}^{\infty} \epsi'_1(t) \diff{t}=\epsi_1(1)=1/3$) on obtient
\begin{equation*} 
u|m_1(u)| \ioe \frac{1}{3}  \sup_{t \ioe u}|M(t)| +\frac{8}{3} \quad (u \soe 1) \;. 
\end{equation*}
Pour tout $u$ on a $u|m(u)| \ioe u|m_1(u)| + |M(u)|$ et $|M(u)| \ioe \sup_{t \ioe u}|M(t)| $, ainsi
\begin{equation} 
u|m(u)| \ioe \frac{4}{3}  \sup_{t \ioe u}|M(t)| + \frac{8}{3} \quad (u \soe 1) \;.  \label{maj bal}
\end{equation}
Puisque le membre de droite de  la majoration \eqref{maj bal}  est croissant, en passant au supremum pour $u \in [1,\, x]$ on arrive à
\begin{equation*} 
\sup_{u \ioe x} u|m(u)| \ioe \frac{4}{3}  \sup_{t \ioe x}|M(t)| + \frac{8}{3} \quad (x \soe 1) \;. 
\end{equation*}
Mais puisque $M(1637)=-16$ on en déduit que pour tout $x\soe 1637$ on a la majoration  $8/3 \ioe (1/6)  \sup_{t \ioe x}|M(t)|$ et donc 
\begin{equation*} 
\sup_{u \ioe x} u|m(u)| \ioe \frac{3}{2}  \sup_{t \ioe x}|M(t)|  \quad (x \soe 1637) \;. 
\end{equation*}

Par le corollaire \ref{autre sens}  on obtient de la même manière
\begin{equation*} 
\sup_{u \ioe x} |M(u)| \ioe 1.44 \sup_{t \ioe x}|m(t)|t + 2 \quad (x \soe 1) \;. 
\end{equation*} 
Mais puisque $m(8510)8510 > 36$ on en déduit que pour tout $x\soe 8510$ on a la majoration  $2 \ioe (1.5-1.44)\sup_{t \ioe x}|m(t)|t $ et donc
\begin{equation*} 
\sup_{t \ioe x} |M(t)| \ioe \frac{3}{2}  \sup_{t \ioe x}|m(t)|t  \quad (x \soe 8510) \;. 
\end{equation*}

Une vérification directe avec le logiciel PARI/GP donne l'encadrement de l'énoncé entre $94$ et $8510$.
\end{proof}

\bibliographystyle{alpha}

\bibliography{bibarticle}

\end{document}